\documentclass[12 pt]{amsart}
\usepackage{amssymb, mathtools, amsfonts, amsthm, graphics,mathrsfs, mathtools}
\usepackage[usenames, dvipsnames]{xcolor}
\usepackage[hmargin=1 in, vmargin = 1 in]{geometry}
\usepackage{tikz-cd}
\usetikzlibrary{matrix, calc, arrows} 
\usepackage{hyperref}
\usepackage[all]{xy}
\usepackage{ytableau}
\usepackage{soul,dirtytalk}
\usepackage[capitalise, noabbrev]{cleveref}
\usepackage{graphicx,pgf,tikz}

\usepackage{twemojis}
\usepackage{mathdots}

\definecolor{darkblue}{rgb}{0.0,0,0.7}

\newcommand{\newword}[1]{\textcolor{darkblue}{\textbf{\emph{#1}}}}

\newtheorem{Theorem}{Theorem}[section]

\newtheorem{corollary}[Theorem]{Corollary}
\newtheorem{conjecture}[Theorem]{Conjecture}

\newtheorem{remark}[Theorem]{Remark}

\newtheorem{definition}[Theorem]{Definition}
\newtheorem{proposition}[Theorem]{Proposition}
\newtheorem{prop}[Theorem]{Proposition}

\newtheorem{theorem}[Theorem]{Theorem}
\newtheorem{theorem*}{Theorem}

\newtheorem{lemma}[Theorem]{Lemma}

\numberwithin{equation}{section}

\newcommand{\rk}{\mathrm{rk}}
\newcommand{\des}{\mathrm{des}}
\newcommand{\ess}{\mathrm{ess}}
\newcommand{\Mat}{\mathrm{Mat}}
\newcommand{\perm}{\mathrm{Perm}}

\newcommand{\init}{\mathrm{in}}

\newcommand{\inv}{\mathrm{inv}}
\newcommand{\codim}{\mathrm{codim}}

\newcommand{\chains}{\mathrm{chains}}

\newcommand{\maj}{\mathrm{maj}}

\newcommand{\Z}{\mathbb{Z}}

\newcommand{\hilb}{\mathrm{Hilb}}
\newcommand{\field}{\kappa}

\theoremstyle{remark}
\newtheorem{examplex}[Theorem]{Example}
\newenvironment{example}
  {\pushQED{\qed}\examplex}
  {\popQED\endexamplex}

\usepackage[colorinlistoftodos]{todonotes}

\newcommand{\asm}{{\sf ASM}}

\newcommand{\anti}{{\sf anti}}

\title{Algebra and geometry of ASM weak order}

\author{Laura Escobar}
\address[LE]{Dept.\ of Mathematics, University of California, Santa Cruz CA 95064}
\email{lauraescobar@ucsc.edu}

\author{Patricia Klein}
\address[PK]{Department of Mathematics, Texas A\&M University, College Station TX 77840}
\email{pjklein@tamu.edu}

\author{Anna Weigandt}
\address[AW]{School of Mathematics, University of Minnesota, Minneapolis MN 55455}
\email{weigandt@umn.edu}

\thanks{LE was partially funded by NSF CAREER grant DMS-2142656. PK was partially funded by NSF grant DMS-2246962 and by the Travel Support for Mathematicians gift MP-TSM-00002939 from the Simons Foundation. AW was partially funded by NSF grant DMS-2344764. PK worked on this project while she was a member at the Institute for Advanced Study with support from the Bob Moses Fund. She thanks the IAS for its hospitality and support.}

\date{\today}
\keywords{}

\makeatletter
\@namedef{subjclassname@2020}{%
  \textup{2020} Mathematics Subject Classification}
\makeatother

\subjclass[2020]{}

\begin{document}

\begin{abstract}
Much of modern Schubert calculus is centered on Schubert varieties in the complete flag variety $\mathcal{F}\ell(n)$ and on their classes in the integral cohomology ring $\mathcal{H}^\ast(\mathcal{F}\ell(n))$. Under the Borel isomorphism, these classes are represented by distinguished polynomials called Schubert polynomials, introduced by Lascoux and Sch\"utzenberger.  Schubert polynomials form an additive basis of $\mathcal{H}^\ast(\mathcal{F}\ell(n))$.

Knutson and Miller showed that Schubert polynomials are multidegrees of matrix Schubert varieties, affine varieties introduced by Fulton, which are closely related to Schubert varieties.  Many roads to studying Schubert polynomials pass through unions and intersections of matrix Schubert varieties.  The third author showed that the natural indexing objects of arbitrary intersections of matrix Schubert varieties are alternating sign matrices (ASMs), which had previously enjoyed rich study in enumerative combinatorics and, as the six-vertex ice model, in statistical mechanics.  Intersections of matrix Schubert varieties are now called ASM varieties.  Every ASM variety is expressible as a union of matrix Schubert varieties.

Many fundamental algebro-geometric invariants (e.g., codimension, degree,
and Castelnuovo–Mumford regularity) are well understood combinatorially for matrix Schubert varieties, substantially via the combinatorics of strong Bruhat order on $S_n$.  The extension of strong order to $\asm(n)$, the set of $n \times n$ ASMs, has so far  not borne as much algebro-geometric fruit for ASM varieties.

Hamaker and Reiner proposed an extension of weak Bruhat order from $S_n$ to $\asm(n)$, which they studied from a combinatorial perspective.  In the present paper, we place this work on algebro-geometric footing.  We use weak order on ASMs to give a characterization of codimension of ASM varieties.  We also show that weak order operators commute with K-theoretic divided difference operators and that they satisfy the same derivative formula that facilitated the first general combinatorial computation of Castelnuovo–Mumford regularity of matrix Schubert varieties.  Finally, we build from these results to generalizations that apply to arbitrary unions of matrix Schubert varieties.
\end{abstract}

\maketitle

\section{Introduction}\label{sec:intro}

Schubert varieties in the complete flag variety are central objects of study within Schubert calculus. The complete flag variety $\mathcal{F}\ell(n)$ is the left quotient of $\mbox{GL}(n)$ by the Borel subgroup $B_-$ of lower triangular matrices. 
The Borel subgroup of upper triangular matrices $B_+$ acts on $\mathcal{F}\ell(n)$ on the right by multiplication. The closures of the orbits under this action are called Schubert varieties, which are indexed by permutations $w \in S_n$.  
The classes of Schubert varieties form an additive basis for the integral cohomology ring of the complete flag variety.  Under the Borel isomorphism \cite{Bor53}, representatives of these classes are given by Schubert polynomials $\mathfrak S_w$, which were introduced by Lascoux and Sch\"utzenberger \cite{LS82} and defined directly from the combinatorics of the permutation $w$.  
Similarly, double Schubert polynomials, which can also be defined combinatorially, are representatives of Schubert classes in the torus-equivariant cohomology ring of $\mathcal{F}\ell(n)$.

Knutson and Miller \cite{KM05} showed that Schubert polynomials and double Schubert polynomials also arise as multidegrees of Fulton's matrix Schubert varieties, which are affine varieties closely related to Schubert varieties \cite{Ful92}.  Given a Schubert variety $\mathcal{X}_w$ in $\mathcal{F}\ell(n)$, the associated matrix Schubert variety $X_w$ is the Zariski closure of $B_-wB_+$ in the space of $n \times n$ matrices, where $w$ is interpreted as a permutation matrix.  (For a description via defining equations, see \cref{subsect:preliminaries-definitions}.)  
Because multidegrees are preserved under Gr\"obner degeneration, Schubert polynomials can also be read from Gr\"obner degenerations of matrix Schubert varieties, as Knutson and Miller did by using Bergeron and Billey's \cite{BB93} pipe dreams to index antidiagonal initial varieties of matrix Schubert varieties. Recently, this article's second and third authors \cite{KW23} gave a similar connection between diagonal initial schemes of matrix Schubert varieties (which need not be reduced) and Lam, Lee, and Shimozono's \cite{LLS21} bumpless pipe dreams.  

A primary motivation in modern Schubert calculus is to understand the Schubert structure constants, i.e., the coefficients $c_{u,v}^w$ arising in the expansion $[\mathcal{X}_u][\mathcal{X}_v] = \sum c_{u,v}^w [\mathcal{X}_w]$ in the integral cohomology ring of $\mathcal{F}\ell(n)$.  These coefficients coincide with those in the expansion $\mathfrak S_u \mathfrak S_v = \sum c_{u,v}^w \mathfrak S_w$ of a product of Schubert polynomials as a sum of other Schubert polynomials. Several authors have understood special cases of this problem in terms of pipe dream combinatorics (see, e.g., \cite{Kog00, BB93, BHY19, HP23, Hua23}).

In order to further our understanding of this framework, we are motivated to study unions and intersections of matrix Schubert varieties.  This is true both because these unions and intersections allow for inclusion-exclusion counting of multidegrees and also because unions of matrix Schubert varieties arise naturally from Gr\"obner degenerations of matrix Schubert varieties that are compatible with moves on bumpless pipe dreams \cite{KW23}.  

The third author \cite{Wei17} showed that alternating sign matrices, combinatorial objects with a rich history in enumerative combinatorics (see, e.g., \cite{MRR83,Zei96, Kup96} and citations therein), are the indexing objects of arbitrary intersections of matrix Schubert varieties.  For this reason, we call an intersection of matrix Schubert varieties an alternating sign matrix (ASM) variety.

Our broad goal within this program is to understand multidegrees of arbitrary unions and intersections of matrix Schubert varieties in a manner that allows us to glean information about Schubert structure constants. One step in this program is to understand basic invariants of ASM varieties, such as codimension. Ultimately, one would want to be able to read algebro-geometric invariants of the ASM variety $X_A$ from the ASM $A$ with the same ease that these invariants can be read for a matrix Schubert variety $X_w$ from the permutation $w$, as is done in \cite{Ful92, RRW22, RRRSDW21, PY24}, for example.

Traditionally, the combinatorial tool underpinning algebro-geometric inquiry into matrix Schubert varieties has been strong (Bruhat) order, under which the lattice of ASMs emerges as the MacNeille completion of $S_n$ \cite{LS96}.  In assessing invariants of ASM varieties, this path has proved difficult. In the present work, we lean instead on the combinatorics of weak (Bruhat) order, a tack also taken in \cite{PSW24} in their study of matrix Schubert varieties.  Hamaker and Reiner \cite{HR20} proposed an extension of weak order from $S_n$ to $\asm(n)$, the set of $n \times n$ ASMs, arguing for its naturality for combinatorial reasons.  Terwilliger \cite{Terwilliger18} introduced a poset whose maximal chains are in bijection with ASMs.  Hamaker and Reiner showed that linear extensions of the weak order on $\asm(n)$ give shelling orders on Terwilliger's poset.

In this paper, we present an algebro-geometric case that agrees with \cite{HR20}'s conclusion. Using their notion of ASM weak order, we show the following:

\begin{theorem*}[{\cref{cor:codimension-saturated-chains}}]
    The codimension of the ASM variety $X_A$ is equal to the minimum length of a saturated chain from the identity permutation to $A$ in weak order.
\end{theorem*}

Grothendieck polynomials correspond to representatives of the K-theory classes of Schubert varieties. They are key ingredients in the known combinatorial formulas for Castelnuovo--Mumford regularity of matrix Schubert varieties \cite{RRRSDW21, RRW22, PSW24, DMS24, PY24, PS}.  Grothendieck polynomials are known to satisfy K-theoretic divided difference operators.  We define ASM Grothendieck polynomials, extending the definition from Grothendieck polynomials indexed by permutations, and show the following: 

\begin{theorem*}[{\cref{prop-DDO}}]
    ASM Grothendieck polynomials satisfy $K$-theoretic divided difference recurrences that are compatible with weak order on $\asm(n)$.
\end{theorem*}

We also show that weak order operators respect the components of an ASM variety (\cref{prop:intersectoperator}), that saturated chains in weak order detect the components of ASM varieties (\cref{cor:codimension-saturated-chains}), and that saturated chains in weak order reflect equidimensionality of ASM varieties within a certain set (\cref{prop:saturated-chains-same-length}).  We conjecture a similar statement on the Cohen--Macaulay property (\cref{conj:characterize-CM-by-chains}).  In \cref{prop:derivatives}, we show that ASM Grothendieck polynomials satisfy the same derivative formula that gave rise to the first combinatorial formula for Castelnuovo--Mumford regularity for arbitrary matrix Schubert varieties \cite{PSW24}.  

Finally, in \cref{s:antichains}, we generalize results from the prior sections of the paper on $\asm(n)$ to antichains in $S_n$, which are the indexing objects of arbitrary unions of matrix Schubert varieties.  These proofs use as base cases the results on ASMs themselves. This shows that understanding ASMs not only gives a direct understanding of arbitrary intersections of matrix Schubert varieties but also provides scaffolding to understand their arbitrary unions.

\section{Preliminaries: Definitions and relations among algebra, geometry, and combinatorics}\label{subsect:preliminaries-definitions}

In this section, we will define the set of $n \times n$ alternating sign matrices $\asm(n)$.  By identifying permutations with their permutation matrices, we may view the symmetric group $S_n$ as a subset of $\asm(n)$.  We will define both strong (Bruhat) order and weak (Bruhat) order on $S_n$ and on $\asm(n)$ from both an algebro-geometric perspective and from a combinatorial perspective. Our focus throughout will be the relationships between these two perspectives. 

Given $m\le n\in \mathbb Z$, let $[n]=\{1,\ldots,n\}$ and let $[m,n] = \{m, m+1, \ldots, n-1,n\}$.  Fix throughout this paper an arbitrary field $\kappa$.

\subsection{Alternating sign matrices}

An $n \times n$ \newword{alternating sign matrix (ASM)} $A = (A_{i,j})$ is an $n \times n$ matrix with entries in $\{-1,0,1\}$ so that, \begin{enumerate}
    \item for all $i,m \in [n]$, $\sum_{j=1}^m A_{i,j} \in \{0,1\}$,
    \item for all $j, m \in [n]$, $\sum_{i=1}^m A_{i,j} \in \{0,1\}$, and 
    \item $\sum_{(i,j) \in [n] \times [n]} A_{i,j} = n$.
\end{enumerate} These three conditions are equivalent to the conditions that the nonzero entries in each row and column alternate in sign and sum to one.  We write $\asm(n)$ for the set of $n\times n$ alternating sign matrices.

The symmetric group $S_n$ is the group of bijections from $[n]$ to $[n]$ under composition.  We may also view $S_n$ as the subset of $\asm(n)$ consisting of the \newword{permutation matrices}, i.e., the subset of $\asm(n)$ with entries in $\{0,1\}$. Specifically, if $w\in S_n$, we will associate to $w$ the permutation matrix $M_w$ with $1$'s in positions $(i,w(i))$ for all $i\in [n]$ and $0$'s elsewhere.  We caution the reader that for $v, w\in S_n$, we have $M_{vw} = M_wM_v$.  Identifying an automorphism of $[n]$ with its permutation matrix, we will refer to either as an element of $S_n$.  Given $w\in S_n$ we write $\ell(w)=|\{(i,j):i<j \text{ and } w(i)>w(j)\}|$ for the \newword{Coxeter length} of $w$.

\subsection{Alternating sign matrix varieties and strong order}

To $A \in\asm(n)$, we associate its \newword{corner sum function} $\rk_A$ defined by $\rk_A(i,j)=\sum_{a\in[i],b\in[j]} A_{a,b}$ for all $i,j\in[n]$.  We will often conflate $\rk_A$ with the matrix $(\rk_A(i,j))_{i,j\in[n]}$.  It will also be convenient to  define $\rk_A(i,j)=0$ if $i=0$ or $j=0$.  We will use $\rk_A$ to denote the corner sum function of $A$, whether we are viewing its domain as $[0,n]^2$ or $[n]^2$ and trust that the appropriate domain will be clear from context.

To each $A \in \asm(n)$, there is an associated subvariety of affine $n^2$-space constructed using $\rk_A$.  Let $\Mat(n)$ denote the set of $n\times n$ matrices with coefficients in the field $\field$.  Given $M\in \Mat(n)$ we write $M_{[i],[j]}$ for the restriction of $M$ to the first $i$ rows and $j$ columns.
The \newword{ASM variety} of $A\in\asm(n)$ is 
\[
X_A=\{M\in \Mat(n) : \rk(M_{[i],[j]})\le \rk_A(i,j) \text{ for all } i,j\in[n]\}.
\]
We let $z_{i,j}$ with $i,j\in[n]$ denote the coordinates of $\Mat(n)$, $S=\field [z_{i,j} : i,j\in[n]]$, and $Z=(z_{i,j})_{i,j\in[n]}$.
Let $I_k(Z_{[i],[j]})\subset S$ be the ideal generated by the $k$-minors of $Z_{[i],[j]}$.  We define
\[
I_A = \sum_{i,j\in[n]}I_{\rk_A(i,j)+1}(Z_{[i],[j]}), 
\] which we call the \newword{ASM ideal} of $A$.  Then $X_A = \mathbb V(I_A)$.  By \cite[Section 7.2]{Knu09} and \cite[Lemma 5.9]{Wei17}, or \cite[Lemma 2.6]{KW23}, $I_A$ is radical.

When $A \in S_n$, $X_A$ is called a \newword{matrix Schubert variety}, and $I_A$ is called a \newword{Schubert determinantal ideal}.  Matrix Schubert varieties were introduced by Fulton \cite{Ful92}, who proved, for $A \in S_n$, that $X_A$ is irreducible and Cohen--Macaulay, that $\codim(X_A) = \ell(A)$, and that $I_A$ is a prime ideal defining $X_A$ \cite[Theorem 3.3]{Ful92}.  For further information on matrix Schubert varieties and their connection to Schubert varieties, we refer the reader to \cite{Ful92, KM05}.  

We define \newword{strong (Bruhat) order} on $\asm(n)$ by the relation $A \leq B$ if and only if $X_A\supseteq X_B$. We may also use corner sum functions to define strong order on $\asm(n)$.  We record these equivalent definitions below.

\begin{proposition}
\label{proposition:asmorderfacts}
    Let $A,B\in \asm(n)$.  The following are equivalent:
    \begin{enumerate}
        \item $X_A\supseteq X_B$.
        \item $I_A\subseteq I_B$.
        \item $\rk_A(i,j)\geq \rk_B(i,j)$ for all $i,j\in [n]$.
    \end{enumerate}
When these equivalent conditions are satisfied, we say $A\leq B$ in strong order.
\end{proposition}

If $A, B \in S_n$, the strong order relation $A \leq B$ may also be described purely combinatorially. In order to state these equivalent characterizations, we recall some basic facts about the combinatorics of the symmetric group.

The symmetric group $S_n$ is generated by the \newword{simple reflections} $\{s_1,\ldots, s_{n-1}\}$, where $s_i\in S_n$ exchanges $i$ and $i+1$ and leaves all other inputs fixed.  The Coxeter length of $w\in S_n$ is equal to the least $k\in \mathbb Z_{\geq 0}$ so that $w$ may be written as a product of $k$ simple reflections. A \newword{reduced word} for $w\in S_n$ is a sequence $({i_1},\ldots, i_{\ell(w)})$ such that  $w=s_{i_1}\cdots s_{i_{\ell(w)}}$.  Then $s_{i_1}\cdots s_{i_{\ell(w)}}$ is called a \newword{reduced expression} for $w$.

Given a permutation, we can represent it in either one-line notation or matrix form. The \newword{one-line notation} of $w\in S_n$ is the sequence $w(1)w(2)\cdots w(n)$, where, as is common, we omit parentheses and commas.

We say $u\leq w$ in \newword{strong (Bruhat) order} if some (equivalently, every) reduced word for $w$ contains a reduced word for $u$ as a (not necessarily consecutive) substring; i.e., if $(i_1,\ldots,i_{\ell(w)})$ is a reduced word for $w$, then there is $1\leq j_1<j_2<\cdots<j_{\ell(u)}\leq \ell(w)$ so that $(i_{j_1},\ldots,i_{j_\ell(u)})$ is a reduced word for $u$. The covering relations for strong order are of the form $u\leq w$ if $w = ut_{i,j}$ for some transposition $t_{i,j}=(i\, j)$ and $\ell(u)+1=\ell(w)$.  

From this combinatorial perspective, we could also characterize $\asm(n)$ under strong order as arising as from (specifically, as the MacNeille completion of) $S_n$ under strong order \cite{LS96}. 

For the next propositions we need the concepts of meet and join. Given $A_1, \ldots, A_k \in \asm(n)$, the \newword{join} of $\{A_1, \ldots, A_k\}$, denoted $A_1\vee \cdots \vee A_k$, or $\bigvee_{i \in [k]} A_i$, is the least upper bound of $\{A_1, \ldots, A_k\}$, i.e., 
\begin{equation*}
    A_1\vee \cdots \vee A_k=\min\{B\in \asm(n): B\ge A_i \mbox{ for all } i \in[k]\}.
\end{equation*}
The \newword{meet} of $\{A_1, \ldots, A_k\}$, denoted by $A_1 \wedge \cdots \wedge A_k$, or $\bigwedge_{i\in[k]} A_i$, is the greatest lower bound of $\{A_1, \ldots, A_k\}$, i.e., 
\begin{equation*}
   A_1 \wedge \cdots \wedge A_k=\max\{B\in \asm(n): B\le A_i \mbox{ for all } i \in[k] \}.
\end{equation*}
A \newword{lattice} is a poset for which every pair of elements has a well-defined join and meet.  A lattice is \newword{complete} if every subset of elements has a well-defined join and meet.  Necessarily, finite lattices are complete.  

\begin{proposition}[{\cite[Proposition 5.4]{Wei17}}]
\label{prop:idealsum}
    If $A,B_1, \ldots, B_k \in \asm(n)$ and $A=B_1\vee \cdots \vee B_k$, then
    \begin{enumerate}
        \item $I_A=I_{B_1}+\cdots+I_{B_k}$.
        \item $X_A=X_{B_1}\cap \cdots \cap X_{B_k}$.
    \end{enumerate}
\end{proposition}

Although \Cref{prop:idealsum} shows intersections of arbitrary ASM varieties are again ASM varieties, arbitrary unions of ASM varieties need not be.  For instance, $X_{213}\cup X_{132}$ is not an ASM variety. However, there is an analogous statement of \Cref{prop:idealsum} for unions of ASM varieties that do happen to be another ASM variety.

\begin{proposition}
\label{prop:intersection-of-asm-ideals}
    If $A,B_1, \ldots, B_k \in \asm(n)$, $A=B_1\wedge \cdots \wedge B_k$, and $X_{B_1}\cup \cdots \cup X_{B_k}$ is an ASM variety, then
    \begin{enumerate}
        \item $I_A=I_{B_1}\cap \cdots \cap I_{B_k}$.
        \item $X_A=X_{B_1}\cup \cdots \cup X_{B_k}$.
    \end{enumerate}
\end{proposition}
\noindent \cref{prop:intersection-of-asm-ideals} is immediate from the definitions of meet, union, and strong order.

For $A \in \asm(n)$, let \[
\perm(A)= \{w\in S_n: w\ge A, \text{ and, if } w\ge v\ge A \text{ for some } v\in S_n,\text{ then }w=v\},
\] which we call the \newword{permutation set of $A$}.

\begin{proposition}\cite[Proposition 5.4]{Wei17}\label{prop:I_A-intersection-of-schubs-in-perm}
Let $A \in \asm(n)$.  Then $I_A = \bigcap_{w \in \perm(A)} I_w$ is the minimal prime decomposition of the radical ideal $I_A$. Consequently, $\codim(X_A) = \min\{\ell(w) : w \in \perm(A)\}$.
\end{proposition}

\subsection{Weak (Bruhat) order}\label{subsect:weak-Bruhat-order-definitions}

We now recall facts about weak order on $\asm(n)$, which was introduced by Hamaker and Reiner in \cite{HR20}.  Hamaker and Reiner gave their definition in terms of objects called monotone triangles, which are combinatorially equivalent to ASMs. (In \cite{EKW}, the authors will develop further connections between the presentation in the current paper and monotone triangles.)  For our purposes, we will find it more convenient to work in terms of ASMs and corner sum functions.

In order to define weak order, we introduce a sublattice of $\asm(n)$.

\begin{definition}
\label{def:pi_i}
    Given $A\in \asm(n)$ and $i\in[n-1]$, let 
    \[\pi_i(A)=\min\{B\in \asm(n):\rk_A(a,b)=\rk_B(a,b) \text{ for all } a,b\in [n] \text{ with } a\neq i\}.\]
\end{definition}

It is not obvious from \cref{def:pi_i} that $\pi_i$ is well-defined, i.e., that a unique minimum of the given set necessarily exists. We delay the proof of this fact until \cref{subsect:preliminaries-basic-lemmas}, where it will appear as \cref{lemma:sublattice}.  More direct means for computing $\pi_i(A)$ are developed in \cite{EKW}.

When $w\in S_n$, then $\pi_i(w)=\begin{cases} w & \text{ if } w(i)<w(i+1)\\
ws_i & \text{ if } w(i)>w(i+1)\end{cases}$ (see \cite[Remark~3.3]{HR20}). 

\begin{example}\label{ex:first-pi_i-example}
    Consider the matrices $A, B \in \asm(5)$ satisfying $w = 31524 > A > B > 31254 = ws_3\in \asm(5)$, written below with $0$ entries recorded as blank squares.  Then $31254 = \pi_3(31254) = \pi_3(B) = \pi_3(A) = \pi_3(31524)$.  The reader may wish to keep this example in mind when considering the sublattice discussed in \cref{lemma:sublattice} and when considering the relationship between strong order and weak order operators in \cref{cor:pi_iorderpreserving}.  We will also return to it in \cref{ex:new-ASM-from-increasing-at-essential-cell}.
    
    \[
\raisebox{1.4cm}{$w$: }\begin{tikzpicture}[x=1.5em,y=1.5em]
\draw[step=1,gray, thin] (0,0) grid (5,5);
\draw[color=black, thick](0,0)rectangle(5,5);
\node at (0.5,3.5){$1$};
\node at (1.5,1.5){$1$};
\node at (2.5,4.5){$1$};
\node at (3.5,0.5){$1$};
\node at (4.5,2.5){$1$};
\end{tikzpicture} \hspace{3cm} 
\raisebox{1.4cm}{$A$: }\begin{tikzpicture}[x=1.5em,y=1.5em]
\draw[step=1,gray, thin] (0,0) grid (5,5);
\draw[color=black, thick](0,0)rectangle(5,5);
\node at (0.5,3.5){$1$};
\node at (1.5,1.5){$1$};
\node at (2.5,4.5){$1$};
\node at (3.5,0.5){$1$};
\node at (3.5,1.5){$-1$};
\node at (3.5,2.5){$1$};
\node at (4.5,1.5){$1$};
\end{tikzpicture}
\]  \[
\raisebox{1.4cm}{$B$: }\begin{tikzpicture}[x=1.5em,y=1.5em]
\draw[step=1,gray, thin] (0,0) grid (5,5);
\draw[color=black, thick](0,0)rectangle(5,5);
\node at (0.5,3.5){$1$};
\node at (1.5,2.5){$1$};
\node at (2.5,4.5){$1$};
\node at (2.5,1.5){$1$};
\node at (2.5,2.5){$-1$};
\node at (3.5,0.5){$1$};
\node at (3.5,1.5){$-1$};
\node at (3.5,2.5){$1$};
\node at (4.5,1.5){$1$};
\end{tikzpicture} \hspace{2.6cm} 
\raisebox{1.4cm}{$ws_3$: }\begin{tikzpicture}[x=1.5em,y=1.5em]
\draw[step=1,gray, thin] (0,0) grid (5,5);
\draw[color=black, thick](0,0)rectangle(5,5);
\node at (0.5,3.5){$1$};
\node at (1.5,2.5){$1$};
\node at (2.5,4.5){$1$};
\node at (3.5,0.5){$1$};
\node at (4.5,1.5){$1$};
\end{tikzpicture} \qedhere
\] 
\end{example}

We say that $A\in \asm(n)$ has a \newword{descent} at $i \in [n-1]$ (or that $i$ is a descent of $A$) if $\pi_i(A)\neq A$.  
    Note that in this case,  $\pi_i(A)<A$.  If $\pi_i(A) = A$ or if $i=n$, then we say that $A$ has an \newword{ascent} at $i$ (or that $i$ is an ascent of $A$). We define \newword{weak order} $\preceq$ on $\asm(n)$ to be the transitive closure of the covering relations $\pi_i(A)\prec A$ with $i$ a descent of $A$.

Note that weak order refines strong order (\cite[Remark 3.5]{HR20}), from which it follows that the transitive closure is well-defined.  The weak order operators $\pi_i$ satisfy commutation and braid relations \cite[Proposition 3.2]{HR20}.  That is $\pi_i \circ \pi_j = \pi_j \circ \pi_i$ when $|i-j|>1$ and $\pi_i \circ \pi_{i+1} \circ \pi_i = \pi_{i+1} \circ \pi_i \circ \pi_{i+1}$.  Furthermore, they are idempotent, i.e., $\pi_i^2 = \pi_i$.

As was the case with strong order, we may characterize weak order on $S_n$ purely in terms of the combinatorics of $S_n$, without reference to corner sum functions.  Given $w,v\in S_n$ we say that $u\preceq w$ if $w=uv$ for some $v\in S_n$ so that $\ell(w)=\ell(u)+\ell(v)$.  We call $\preceq$ the \newword{weak (Bruhat) order} on $S_n$.  Equivalently, $u\preceq w$ if there is some reduced word $(i_1,\ldots, i_{\ell(w)})$ for $w$ so that $(i_1,\ldots,i_{\ell(u)})$ is a reduced word for $u$.  Note that this is the transitive closure of the covering relations $u\preceq w$ if $w=us_i$ for some $i\in[n-1]$ and $\ell(w)=\ell(u)+1$. 

$S_n$ under weak order already forms a lattice \cite{Bj84}. For that reason, we cannot view weak order on $\asm(n)$ as canonically induced from weak order on $S_n$ in the same way that we saw strong order on $\asm(n)$ as induced from strong order on $S_n$.  However, Hamaker and Reiner \cite{HR20} gave a combinatorial case that the extension of weak order from $S_n$ to $\asm(n)$ that they choose, which is the one that we use here, is the natural one. For algebro-geometric reasons, we agree.  The present paper constitutes our argument for the position.

\subsection{Preliminaries: Functional Lemmas}\label{subsect:preliminaries-basic-lemmas}

The map $A\mapsto \rk_A$ is injective.  Robbins and Rumsey characterized the corner sum functions that arise from ASMs:
\begin{lemma}[{\cite[Lemma~1]{RR86}}]
\label{lemma:cornerincrease}
    Let $M:[0,n] \times [0,n] \rightarrow \mathbb{Z}$ be a function. There exists $A\in \asm(n)$ such that $M=\rk_A$ if and only if $M$ satisfies:
    \begin{enumerate}
        \item $M(i,0)=M(0,i) = 0$ for all $i \in [n]$,
        \item $M(i,n)=M(n,i)=i$ for all $i\in[n]$, and
        \item $M(i,j)-M(i-1,j),M(i,j)-M(i,j-1)\in\{0,1\}$ for all $i,j\in [n]$.
    \end{enumerate}
\end{lemma}

Under strong order, $\asm(n)$ forms a complete lattice. Indeed, Lascoux and Sch\"utzenberger \cite[Lemme 5.4]{LS96} showed that $\asm(n)$ is the smallest lattice which contains the symmetric group under strong order as a subposet.
 In particular, we have the following:
\begin{lemma}
\label{lemma:joinmeet}
Fix $U\subseteq \asm(n)$.  Write $B=\bigvee_{A\in U}A$ and $C=\bigwedge_{A\in U}A$.  Then
\[\rk_B(i,j)=\min\{\rk_A(i,j):A\in U\}\]
and 
\[\rk_C(i,j)=\max\{\rk_A(i,j):A\in U\}\]
for all $i,j\in[n]$.
\end{lemma}

\begin{proof}
    By induction, we assume $|U| = 2$.  One may verify that, for $A_1,A_2\in \asm(n)$, $M=(\min\{\rk_{A_1}(i,j),\rk_{A_2}(i,j)\})_{i,j=1}^n$ and $N = (\max\{\rk_{A_1}(i,j),\rk_{A_2}(i,j)\})_{i,j=1}^n$ satisfy the three conditions of \cref{lemma:cornerincrease}, from which it follows that there exist $\widetilde{B}, \widetilde{C}\in \asm(n)$ so that $\rk_{\widetilde{B}}=M$ and $\rk_{\widetilde{C}}=N$.  
    It is clear that $\widetilde{B} \geq A_1, A_2$ and that no $B'<\widetilde{B}$ satisfies $B' \geq A_1,A_2$. Similarly, $\widetilde{C} \leq A_1, A_2$ and no $C'>\widetilde{C}$ satisfies $C' \leq A_1, A_2$.  Hence, $B = \widetilde{B}$ and $C = \widetilde{C}$.
\end{proof}

We are now prepared to give the lemma required to see that \cref{def:pi_i} is well-defined.

\begin{lemma}
\label{lemma:sublattice}
    Let $A\in \asm(n)$ and fix $i\in[n-1]$.  The set 
    \[U=\{B\in \asm(n):\rk_A(a,b)=\rk_B(a,b) \text{ for all } a,b\in [n] \text{ with } a\neq i\}\] is a sublattice of $\asm(n)$.
\end{lemma}
\begin{proof}
    It is enough to verify that joins and meets of any two elements in $U$ are also in $U$.  This follows immediately by the definition of $U$ and \cref{lemma:joinmeet}.
\end{proof}

Let $A\in\asm(n)$. We say that $(i,j)$ is an \newword{inversion} of $A$ if 
\[\sum_{k=1}^jA_{i,k}=\sum_{l=1}^i A_{l,j}=0.\]
We write $\inv(A)$ for the set of inversions of $A$.
When $A \in S_n$, then $|\inv(A)|$ is the Coxeter length (also known as the inversion number) of $A$.  Using the term \say{inversion} for ASMs in this way was first done in \cite{MRR83} and is now standard in the literature.

We introduce notation for the set of positions of $-1$'s of $A$, namely 
\begin{equation*}
    N(A)=\{(i,j): A_{i,j}=-1\}
.\end{equation*}
Notice that $N(A)\subseteq \inv(A)$.
The \newword{Rothe diagram} of $A\in\asm(n)$ is
\[
D(A)=\{(i,j) : (i,j) \text{  is an inversion of } A \text{ and } A_{i,j}=0\}.
\]
In particular, $\inv(A)=N(A)\sqcup D(A)$.

We remark that when $A \in S_n$, $|D(A)| = \ell(A) = \codim(X_A)$. However, for $A \in \asm(n) - S_n$, necessarily $D(A) \subsetneq \inv(A)$, and $\codim(X_A)$ may be less than $|D(A)|$, greater than $|\inv(A)|$, or between the two numbers (see, e.g., the examples at the top of Page 21 of \cite{HR20}, interpreted through the lens of \cref{cor:codimension-saturated-chains}).  

\begin{remark}\label{rmk:codim-not-bounded-as-expected} It is, however, quite often the case that $|D(A)| \leq \codim(X_A)$; precisely, the inequality holds for all elements of $\asm(n)$ for $n\leq 7$ and for all but $9$ of the $10,850,216$ elements of $\asm(8)$.  One of those $9$ elements is \[
A = \left(\!\begin{array}{cccccccc}
      0&0&0&0&1&0&0&0\\
      0&0&0&1&0&0&0&0\\
      0&0&1&0&0&0&0&0\\
      0&1&0&0&-1&0&0&1\\
      1&0&0&-1&0&0&1&0\\
      0&0&0&0&0&1&0&0\\
      0&0&0&0&1&0&0&0\\
      0&0&0&1&0&0&0&0
      \end{array}\!\right),
      \] which satisfies $|D(A)| = 17$ and $\codim(X_A) = 16$.
\end{remark}

The \newword{essential set} of $A$ is 
\[
\ess(A)=\{(i,j)\in \inv(A) : (i + 1, j), (i, j + 1) \notin \inv(A)\}.
\]
We call an element of $\ess(A)$ an essential cell of $A$.

We will see a poset theoretic significance of the essential set in \cref{lem:cover-by-add-to-ess-cell}.  The essential set derives its name from Fulton's result \cite[Lemma 3.10]{Ful92} in that it records the rank conditions that are essential for determining Schubert determinantal ideals, a result extended to all ASM ideals by the third author \cite[Lemma 5.9]{Wei17}.  Specifically, \[
I_A= \sum_{(i,j)\in \ess(A)}I_{\rk_A(i,j)+1}(Z_{[i],[j]}).
\] 

\begin{example}
\label{example:rothe}
Let $A=\begin{pmatrix} 
    0 & 0 & 1 & 0\\
    1 & 0  & -1 & 1 \\
    0 & 1 & 0 & 0\\
    0 & 0 & 1 & 0
    \end{pmatrix}.$  Then $D(A) = \{(1,1),(1,2)\}$, $N(A)=\{(2,3)\}$, $\inv(A) = D(A) \cup N(A)$, and $\ess(A) = \{(1,2),(2,3)\}$.  To visualize these sets, we draw a picture as below.  See \cite[Section 3.1]{Wei21} for a detailed explanation.  By a mild abuse of notation, we will tend to refer to this picture as the Rothe diagram.
    
    
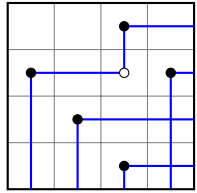
\[
   \begin{tikzpicture}[x=1.5em,y=1.5em]
\draw[step=1,gray, thin] (0,0) grid (4,4);
\draw[color=black, thick](0,0)rectangle(4,4);

\draw[thick, color=blue] (.5,0)--(.5,2.5)--(2.4,2.5);
\draw[thick, color=blue] (2.5,2.6)--(2.5,3.5)--(4,3.5);
\draw[thick, color=blue] (1.5,0)--(1.5,1.5)--(4,1.5);
\draw[thick, color=blue] (2.5,0)--(2.5,0.5)--(4,0.5);
\draw[thick, color=blue] (3.5,0)--(3.5,2.5)--(4,2.5);
\filldraw [black](0.5,2.5)circle(.1);
\filldraw [black](1.5,1.5)circle(.1);
\filldraw [black](2.5,0.5)circle(.1);
\draw [black](2.5,2.5)circle(.1);
\filldraw [black](2.5,3.5)circle(.1);
\filldraw [black](3.5,2.5)circle(.1);
\end{tikzpicture} 
\] 

Note \[
I_A = I_1\begin{pmatrix} z_{1,1} & z_{1,2}\end{pmatrix}+I_2\begin{pmatrix} z_{1,1} & z_{1,2} & z_{1,3} \\ z_{2,1} & z_{2,2} & z_{2,3}
\end{pmatrix}. \qedhere
\]
\end{example}

The following lemma gives a description of the inversion set and essential set in terms of the rank function.

\begin{lemma}\label{lem:Ess(A)ByRanks}
    Let $A \in \asm(n)$ and $(i,j) \in [n] \times [n]$.  Then \begin{enumerate} 
    \item $(i,j) \in \inv(A)$ if and only if 
\[
\rk_A(i,j)=\rk_A(i,j-1)=\rk_A(i-1,j),\mbox{ and }
\]
    \item $(i,j) \in \ess(A)$ if and only if 
\[
\rk_A(i,j)=\rk_A(i,j-1)=\rk_A(i-1,j)=\rk_A(i,j+1)-1=\rk_A(i+1,j)-1.
\]
\end{enumerate}
\end{lemma}

\begin{proof}
\noindent (1) This follows from \cite[Lemma 3.5]{Wei21}.    

\noindent (2) This is immediate from part (1) and the definition of $\ess(A)$. 
\end{proof}

A permutation is \newword{bigrassmannian} if it has a unique essential cell.  There is a bijection between bigrassmannian permutations in $S_n$ and triples $(i,j,r)$ which satisfy the following conditions:
\begin{enumerate}
    \item $1\leq i,j$,
    \item $0\leq r<\min(i,j)$, and 
    \item $i+j-r\leq n$.
\end{enumerate}
We map such a triple to the unique bigrassmannian permutation $u\in S_n$ with $\ess(u)=\{(i,j)\}$ and $\rk_u(i,j)=r$.  We will write $[(i,j),r]_b$ for the bigrassmannian permutation mapped to by the triple $(i,j,r)$.  
Specifically, if $\mathbf{1}_k$ denotes a $k \times k$ identity matrix, then the bigrassmannian $u$ is, in block matrix form, \[\begin{pmatrix}
\mathbf{1}_r&\mathbf 0&\mathbf 0&\mathbf 0\\
\mathbf 0 &\mathbf 0& \mathbf{1}_{i-r}& \mathbf 0\\
\mathbf 0&\mathbf{1}_{j-r} &\mathbf 0 &\mathbf 0\\
\mathbf 0&\mathbf 0&\mathbf 0&\mathbf{1}_{n-i-j+r}
\end{pmatrix}.\]

Note that the bigrassmannian permutations are exactly those whose matrix Schubert varieties are classical determinantal varieties, up to affine factors.

\begin{example}
    In $S_5$, we have $[(3,4),2]_b$ is the unique permutation $u$ whose essential set is exactly $\{(3,4)\}$ and which satisfies $\rk_u(3,4) = 2$, i.e., $[(3,4),2]_b=12534$, or, in matrix form, \[[(3,4),2]_b=
\begin{pmatrix}
    1 & 0 & 0 & 0 & 0\\
    0 & 1 & 0 & 0 & 0\\
    0 & 0 & 0 & 0 & 1\\
    0 & 0 & 1 & 0 & 0\\
    0 & 0 & 0 & 1 & 0\\
\end{pmatrix}. \qedhere
    \]
\end{example}

\begin{lemma}
\label{lemma:bigrasscompare}
Let $u\in S_n$ be the bigrassmannian permutation $[(i,j),r]_b$.  Then
\begin{enumerate}
    \item $u=\wedge \{A\in \asm(n):\rk_A(i,j)\leq r\}$, and
    \item if $A\in \asm(n)$ so that $\rk_A(i,j)\leq r$ then $u\leq A$.
\end{enumerate}
\end{lemma}
\begin{proof}
    For a proof of Part (1), see \cite[Theorem 30]{BS17}.  Part (2) is immediate from part (1).
\end{proof}

\begin{lemma}
\label{lemma:asmbigrass}
    Let $A\in \asm(n)$, and let $\{u_1, \ldots, u_k\}\ = \{[(i,j),\rk_A(i,j)]_b:(i,j)\in \ess(A)\}$, i.e., the bigrassmannian permutations corresponding to the essential cells of $A$.  Then \begin{enumerate}
        \item $A=\vee\{u_1, \ldots, u_k\}$, \item $I_A = I_{u_1}+\cdots+I_{u_k}$, and 
        \item $X_A = X_{u_1} \cap \cdots \cap X_{u_k}$.
        
    \end{enumerate}
\end{lemma}
\begin{proof}
    For Part (1), by \cite[Proposition 3.11]{Wei17}, $\{[(i,j),\rk_A(i,j)]_b:(i,j)\in \ess(A)\}$ is set of bigrassmannian permutations which are maximal among those below $A$ in strong order.  Thus, the statement $A=\vee \{[(i,j),\rk_A(i,j)]_b:(i,j)\in \ess(A)\}$ follows from \cite[Proposition 2.4]{LS96} and \cite[Lemma 5.4]{LS96}.
    
    Parts (2) and (3) follow from Part (1) together with \cref{prop:idealsum}.
\end{proof}

\begin{example}
Returning to the ASM $A$ from \cref{example:rothe} and taking $u_1 = 3124$ and $u_2 = 1423$, the reader may verify $I_A = I_{u_1}+I_{u_2}$.  To see that $u_1$ and $u_2$ are both bigrassmannian permutations, note that $3124 = [(1,2),0]_b$ and $1423 = [(2,3),1]_b$.
\end{example}

\begin{lemma}\label{lemma:new-ASM-from-increasing-at-essential-cell}
    Let $A \in \asm(n)$, and suppose that $(i,j) \in \ess(A)$.  Let $f:[n] \times [n] \rightarrow [0,n]$ be the function \[
f(a,b) = \begin{cases}
 \rk_A(a,b) & \text{ if } (a,b) \neq (i,j)\\
 \rk_A(i,j)+1 & \text{ if } (a,b) = (i,j).
\end{cases}
    \] Then there exists $B \in \asm(n)$ so that $\rk_B = f$.  Specifically, \[
B_{a,b} = \begin{cases}
A_{a,b} & \text{ if } (a,b) \notin \{(i,j),(i+1,j),(i,j+1),(i+1,j+1)\} \\
A_{a,b}+1 & \text{ if } (a,b) \in \{(i,j), (i+1,j+1)\}\\
A_{a,b}-1 & \text{ if } (a,b) \in \{(i+1,j), (i,j+1)\}.
\end{cases}
    \]
\end{lemma}

\begin{proof}
This is a simple exercise from the definitions of $\asm$ and corner sum functions.  See, for example, \cite[Lemma 2]{BS17} for further details.
\end{proof}

\begin{example}\label{ex:new-ASM-from-increasing-at-essential-cell}
    Returning to \cref{ex:first-pi_i-example}, note that $\rk_A(a,b) = \rk_w(a,b)$ for all $(a,b) \neq (3,4)$ and that $\rk_A(3,4) = \rk_w(3,4)+1$.  Similarly, $B$ is obtained from $A$ by modifying the rank function at $(3,2)$, and $ws_3$ obtained from $B$ by modifying the rank function at $(3,3)$.  We repeat these four ASMs below, this time recorded via their Rothe diagrams to facilitate identifying their essential cells.

    \[
\raisebox{1.4cm}{$w$: }\begin{tikzpicture}[x=1.5em,y=1.5em]
\draw[step=1,gray, thin] (0,0) grid (5,5);
\draw[color=black, thick](0,0)rectangle(5,5);

\draw[thick, color=blue] (.5,0)--(.5,3.5)--(5,3.5);
\draw[thick, color=blue] (1.5,0)--(1.5,1.5)--(5,1.5);
\draw[thick, color=blue] (2.5,0)--(2.5,4.5)--(5,4.5);
\draw[thick, color=blue] (3.5,0)--(3.5,.5)--(5,.5);
\draw[thick, color=blue] (4.5,0)--(4.5,2.5)--(5,2.5);
\filldraw[black](0.5,3.5)circle(.1);
\filldraw[black](1.5,1.5)circle(.1);
\filldraw[black](2.5,4.5)circle(.1);
\filldraw[black](3.5,0.5)circle(.1);
\filldraw[black](4.5,2.5)circle(.1);
\end{tikzpicture} \hspace{3cm} 
\raisebox{1.4cm}{$A$: }\begin{tikzpicture}[x=1.5em,y=1.5em]
\draw[step=1,gray, thin] (0,0) grid (5,5);
\draw[color=black, thick](0,0)rectangle(5,5);

\draw[thick, color=blue] (.5,0)--(.5,3.5)--(5,3.5);
\draw[thick, color=blue] (1.5,0)--(1.5,1.5)--(3.4,1.5);
\draw[thick, color=blue] (3.5,1.6)--(3.5,2.5)--(5,2.5);
\draw[thick, color=blue] (2.5,0)--(2.5,4.5)--(5,4.5);
\draw[thick, color=blue] (3.5,0)--(3.5,.5)--(5,.5);
\draw[thick, color=blue] (4.5,0)--(4.5,1.5)--(5,1.5);
\filldraw[black](0.5,3.5)circle(.1);
\filldraw[black](1.5,1.5)circle(.1);
\filldraw[black](2.5,4.5)circle(.1);
\filldraw[black](3.5,0.5)circle(.1);
\draw[black](3.5,1.5)circle(.1);
\filldraw[black](3.5,2.5)circle(.1);
\filldraw[black](4.5,1.5)circle(.1);
\end{tikzpicture}
\]  \[
\raisebox{1.4cm}{$B$: }\begin{tikzpicture}[x=1.5em,y=1.5em]
\draw[step=1,gray, thin] (0,0) grid (5,5);
\draw[color=black, thick](0,0)rectangle(5,5);
\draw[thick, color=blue] (.5,0)--(.5,3.5)--(5,3.5);
\draw[thick, color=blue] (1.5,0)--(1.5,2.5)--(2.4,2.5);
\draw[thick, color=blue] (2.5,2.6)--(2.5,4.5)--(5,4.5);
\draw[thick, color=blue] (2.5,0)--(2.5,1.5)--(3.4,1.5);
\draw[thick, color=blue] (3.5,1.6)--(3.5,2.5)--(5,2.5);
\draw[thick, color=blue] (3.5,0)--(3.5,.5)--(5,.5);
\draw[thick, color=blue] (4.5,0)--(4.5,1.5)--(5,1.5);
\filldraw[black](0.5,3.5)circle(.1);
\filldraw[black](1.5,2.5)circle(.1);
\filldraw[black](2.5,4.5)circle(.1);
\filldraw[black](2.5,1.5)circle(.1);
\draw[black](2.5,2.5)circle(.1);
\filldraw[black](3.5,0.5)circle(.1);
\draw[black](3.5,1.5)circle(.1);
\filldraw[black](3.5,2.5)circle(.1);
\filldraw[black](4.5,1.5)circle(.1);

\end{tikzpicture} \hspace{2.6cm} 
\raisebox{1.4cm}{$ws_3$: }\begin{tikzpicture}[x=1.5em,y=1.5em]
\draw[step=1,gray, thin] (0,0) grid (5,5);
\draw[color=black, thick](0,0)rectangle(5,5);
\draw[thick, color=blue] (.5,0)--(.5,3.5)--(5,3.5);
\draw[thick, color=blue] (1.5,0)--(1.5,2.5)--(5,2.5);
\draw[thick, color=blue] (2.5,0)--(2.5,4.5)--(5,4.5);
\draw[thick, color=blue] (3.5,0)--(3.5,.5)--(5,.5);
\draw[thick, color=blue] (4.5,0)--(4.5,1.5)--(5,1.5);
\filldraw[black](0.5,3.5)circle(.1);
\filldraw[black](1.5,2.5)circle(.1);
\filldraw[black](2.5,4.5)circle(.1);
\filldraw[black](3.5,0.5)circle(.1);
\filldraw[black](4.5,1.5)circle(.1);

\end{tikzpicture} \qedhere
\] 
\end{example}

\begin{lemma}\label{lem:cover-by-add-to-ess-cell}
    Let $A,B\in \asm(n)$.  Then $A$ covers $B$ in strong order if and only if there exists $(i,j)\in\ess(A)$ so that $\rk_A(a,b)=\rk_B(a,b)$ for all $(a,b)\neq (i,j)$ and $\rk_A(i,j)+1=\rk_B(i,j)$.  
\end{lemma}

\begin{proof}
\noindent $(\Rightarrow)$ Suppose $A$ covers $B$ in strong order.  If $\rk_A(i,j)=\rk_B(i,j)$ for all $(i,j)\in \ess(A)$, then, by \cref{lemma:bigrasscompare}, we have that $B\geq [(i,j),\rk_A(i,j)]_b$ for all $(i,j)\in \ess(A)$.  Thus, $B\geq \vee \{[(i,j),\rk_A(i,j)]_b:(i,j)\in \ess(A)\}$.  By \cref{lemma:asmbigrass}, 
$$A=\vee \{[(i,j),\rk_A(i,j)]_b:(i,j)\in \ess(A)\}\leq B\leq A.$$
Thus $B=A$, which is a contradiction.

Because $A \geq B$, we know $\rk_A(i,j) \leq \rk_B(i,j)$ for all $(i,j) \in [n] \times [n]$.  Hence, there exists some $(r,s) \in \ess(A)$ so that $\rk_A(r,s) < \rk_B(r,s)$.  Let $C$ be the ASM obtained from $A$ and $(r,s)$ using \cref{lemma:new-ASM-from-increasing-at-essential-cell}.  By construction, $B \leq C < A$, and $C$ has a rank function of the desired form.  Because $A$ covers $B$, $B = C$.  

\noindent $(\Leftarrow)$  Let $A,B\in \asm(n)$  and suppose there exists $(i,j)\in\ess(A)$ so that $\rk_A(a,b)=\rk_B(a,b)$ for all $(a,b)\neq (i,j)$ and $\rk_A(i,j)+1=\rk_B(i,j)$. This implies $A > B$, and it is clear that it is a covering relation. \end{proof}

\section{Components and codimension of $X_A$ in terms of the weak order poset}

In this section, we will relate algebro-geometric properties of ASM varieties to the poset $\asm(n)$ under weak order.  Our main goals are to show that the weak order operators $\pi_i$ respect strong order (\cref{cor:pi_iorderpreserving}) and to give the codimension of an ASM variety in terms of the weak order poset (\cref{cor:codimension-saturated-chains}).  

\subsection{Compatibility of weak order operators and strong order}

\begin{lemma}\label{lem:oneStep}
    Let $A \in \asm(n)$, and suppose that $(i,j)$ is an essential cell of $A$.  Let $B$ be the ASM so that $\rk_{B}(i,j) = \rk_A(i,j)+1$ and $\rk_{B}(a,b) = \rk_A(a,b)$ for all $(a,b) \neq (i,j)$.  Then $\pi_i(A) = \pi_i(B)$.
\end{lemma}

\begin{proof}
By \cref{lemma:new-ASM-from-increasing-at-essential-cell}, the ASM $B$ in the statement of the present lemma exists.
Because $r_A$ and $r_B$ only differ in row $i$,
\begin{align*}
    \{C\in \asm(n):\rk_A(a,b)&=\rk_C(a,b) \text{ for all } a,b\in [n] \text{ with } a\neq i\}=\\
    &\{C\in \asm(n):\rk_B(a,b)=\rk_C(a,b) \text{ for all } a,b\in [n] \text{ with } a\neq i\}.
    \end{align*}
    Thus, by definition, $\pi_i(A)=\pi_i(B)$.
\end{proof}

    For an example, we again direct the reader to \cref{ex:first-pi_i-example} (or \cref{ex:new-ASM-from-increasing-at-essential-cell}).

\begin{lemma}
\label{lemma:descentessential}
    Let $A\in\asm(n)$.  Then $i$ is a descent of $A$ if and only if there is an essential cell in row $i$ of $A$.
\end{lemma}

\begin{proof}
Let $A\in \asm(n)$.
    \noindent $(\Rightarrow)$ Suppose that $i$ is a descent of $A$.  Then $\pi_i(A)<A$, which implies that $\pi_i(A)$ can be obtained from $A$ via a sequence of strong order covers $\pi_i(A) = A_0 < A_1 < \cdots < A_k = A$ with $k \geq 1$. By \cref{lem:cover-by-add-to-ess-cell}, the covering relation $A_{k-1} < A$ requires the existence of some $(a,b) \in \ess(A)$ such that $\rk_{A}(a,b)<\rk_{A_{k-1}}(a,b)$.  Then necessarily $\rk_A(a,b)<\rk_{\pi_i(A)}(a,b)$. If $A$ does not have an essential cell in row $i$, then $a \neq i$ and $\rk_A$ and $\rk_{\pi_i(A)}$ differ outside of row $i$, in violation of the definition of $\pi_i(A)$.

    \noindent $(\Leftarrow)$ Suppose $A$ has an essential cell in row $i$.  By \cref{lem:oneStep}, there exists $B<A$ with $\pi_i(A) = \pi_i(B)$.  Then $\pi_i(A) =\pi_i(B) \leq B <A$, as desired.
\end{proof}

Note that \cref{lemma:descentessential} implies that any $A \in \asm(n)$ with an essential cell in each row from $1$ to $n-1$ is a maximal element with respect to weak order.  For example, the $9$ maximal elements of $\asm(4)$ are \[
 \begin{pmatrix}
    0 & 1 & 0 & 0\\
    1 & -1 & 1 & 0\\
    0 & 1 & -1 & 1\\
    0 & 0 & 1 & 0
\end{pmatrix}, \quad \begin{pmatrix}
    0 & 1 & 0 & 0\\
    0 & 0 & 1 & 0\\
    1 & 0 & -1 & 1\\
    0 & 0 & 1 & 0
\end{pmatrix}, \quad \begin{pmatrix}
    0 & 0 & 1 & 0\\
    0 & 1 & 0 & 0\\
    1 & 0 & -1 & 1\\
    0 & 0 & 1 & 0
\end{pmatrix}, \] \[ \begin{pmatrix}
    0 & 0 & 1 & 0\\
    0 & 1 & 0 & 0\\
    1 & -1 & 0 & 1\\
    0 & 1 & 0 & 0
\end{pmatrix}, \quad \begin{pmatrix}
    0 & 1 & 0 & 0\\
    1 & -1 & 0 & 1\\
    0 & 0 & 1 & 0\\
    0 & 1 & 0 & 0
\end{pmatrix}, \quad
\begin{pmatrix}
    0 & 0 & 1 & 0\\
    0 & 1 & -1 & 1\\
    1 & -1 & 1 & 0\\
    0 & 1 & 0 & 0
\end{pmatrix}, \] \[
\begin{pmatrix}
    0 & 0 & 0 & 1\\
    0 & 1 & 0 & 0\\
    1 & -1 & 1 & 0\\
    0 & 1 & 0 & 0
\end{pmatrix},
\quad
\begin{pmatrix}
    0 & 0 & 1 & 0\\
    0 & 1 & -1 & 1\\
    0 & 0 & 1 & 0\\
    1 & 0 & 0 & 0
\end{pmatrix}, 
\mbox{ and } 
\begin{pmatrix}
    0 & 0 & 0 & 1\\
    0 & 0 & 1 & 0\\
    0 & 1 & 0 & 0\\
    1 & 0 & 0 & 0
\end{pmatrix}.
\]  

This situation stands in contrast to that of weak order on $S_n$, in which the only top element is $w_0=n \, n-1 \, \ldots \, 1$.  

It will be shown in \cite{EKW} that $A \in \asm(n)$ is maximal with respect to weak order if and only if it has an essential cell in each row from $1$ to $n-1$.

\begin{lemma}
\label{lemma:joindescent}
    Suppose $A,B,C\in \asm(n)$ so that $A=B\vee C$ (equivalently, $I_A = I_B+I_C$).  If $i$ is a descent of $A$, then $i$ is a descent of $B$ or $i$ is a descent of $C$.  
\end{lemma}
\begin{proof}
  Suppose $i$ is a descent of $A$.  Then by \cref{lemma:descentessential}, $A$ has an essential cell in row $i$, say at position $(i,j)$.  For convenience, write $\rk_A(i,j)=k$.  Since $(i,j)$ is an essential cell, by \cref{lem:Ess(A)ByRanks} we have $\rk_A(i+1,j)=\rk_A(i,j+1)=k+1$ and $\rk_A(i-1,j)=\rk_A(i,j-1)=k$.
  Because $A=B\vee C$, we know from \cref{lemma:joinmeet} that
  \begin{equation}
      \label{eq:rankmin}
      \rk_A(a,b)=\min\left(\rk_B(a,b),\rk_C(a,b)\right) \text{ for all } a,b\in[n].
  \end{equation}
  Thus, $\rk_B(i,j)=k$ or $\rk_C(i,j)=k$.  Without loss of generality, assume $\rk_B(i,j)=k$.  From \cref{eq:rankmin} and $\rk_A(i+1,j)=\rk_A(i,j+1)=k+1$, we have $\rk_B(i+1,j), \rk_B(i,j+1) \geq k+1$.  From \cref{lemma:cornerincrease}, $\rk_B(i+1,j), \rk_B(i,j+1) \leq k+1$.  Hence $\rk_B(i+1,j)=\rk_B(i,j+1)=k+1$.  Again by \cref{eq:rankmin}, $\rk_B(i-1,j), \rk_B(i,j-1)\geq k$.  But also $\rk_B(i-1,j), \rk_B(i,j-1)\leq \rk_B(i,j) = k$, and so $\rk_B(i-1,j)= \rk_B(i,j-1) = k$. Applying \cref{lem:Ess(A)ByRanks} again, we conclude that $(i,j)$ is an essential cell of $B$ as well.  The result now follows from \cref{lemma:descentessential}.
\end{proof}

\begin{proposition}\label{prop:joindescentpart3}
    Suppose $A,B_1,\ldots,B_k\in \asm(n)$ so that $A=B_1\vee \cdots \vee B_k$.  Let $i\in [n-1]$.  Then $\pi_i(A)=\pi_i(B_1)\vee \cdots \vee \pi_i(B_k)$.  Equivalently, if $I_A = \sum_{j=1}^k I_{B_j}$, then $I_{\pi_i(A)} = \sum_{j=1}^k I_{\pi_i({B_j})}$.
\end{proposition}
\begin{proof}
We proceed by induction on $k$.  If $k=1$ there is nothing to show.  Suppose $k=2$. For convenience, we write $C=\pi_i(B_1)\vee \pi_i(B_2)$.  It follows from \cref{lemma:joinmeet} that $\rk_{\pi_i(A)}$ agrees with $\rk_C$ outside of row $i$.  As such, $\pi_i(C)=\pi_i(A)$.  If $C\neq \pi_i(A)$, then \cref{lemma:joindescent} implies that $i$ is a descent of $\pi_i(B_1)$ or of $\pi_i(B_2)$, which contradicts \cref{lemma:descentessential}.  Thus, $C=\pi_i(A)$ as desired.

Now assume $k>2$.  Then because $A=(B_1\vee  \cdots \vee B_{k-1})  \vee B_k$, we apply the inductive hypothesis twice and obtain
\begin{align*}
    \pi_i(A)&=\pi_i(B_1\vee  \cdots \vee B_{k-1})  \vee \pi_i(B_k)\\
    &=\left(\pi_i(B_1)\vee \cdots \vee \pi_i(B_{k-1})\right)\vee\pi_i(B_k)\\
    &=\pi_i(B_1)\vee \cdots \vee \pi_i(B_k). 
\end{align*}
The final statement follows from \cref{prop:idealsum}.
\end{proof}

\begin{remark}\label{meetCounterexample}
    The statements corresponding to \cref{lemma:joindescent} and \cref{prop:joindescentpart3} pertaining to meets rather than joins are false.  For example, $A = \begin{pmatrix} 
0 & 1 & 0 & 0\\
1 & -1 & 1 & 0\\
0 & 1 & 0 & 0\\
0 & 0 & 0 & 1
    \end{pmatrix} = 2341 \wedge 3124
    $ has a descent at $2$ while neither $2341$ nor $3124$ does.  Then $2134 = \pi_2(A)  \neq \pi_2(2341) \wedge \pi_2(3124) = 2341 \wedge 3124$.
\end{remark}

\begin{theorem}
\label{cor:pi_iorderpreserving}
    Let $A,B\in \asm(n)$ so that $A\leq B$.  Let $i\in[n-1]$.  Then $\pi_i(A)\leq \pi_i(B)$.  Equivalently, if $I_A \subseteq I_B$, then $I_{\pi_i(A)} \subseteq I_{\pi_i(B)}$.
\end{theorem}
\begin{proof}
    Because $A\leq B$,  $A\vee B=B$.  By \cref{prop:joindescentpart3}, we have $\pi_i(A)\vee \pi_i(B)=\pi_i(B)$.  This implies $\pi_i(A)\leq \pi_i(B)$.
\end{proof}

In the case that $A,B\in S_n$, \cref{cor:pi_iorderpreserving} follows by a routine application of the Lifting Property for permutations (\cite[Proposition 2.2.7]{BB05}).

\subsection{Components of $X_{\pi_i(A)}$}

 Our next goal is to write $I_{\pi_i(A)}$ as an intersection of Schubert determinantal ideals, allowing some potential redundancies.

\begin{lemma}\label{lem:permpiAascent}
    Let $A \in \asm(n)$ and $i \in [n-1]$.  If $v \in \perm(\pi_i(A))$, then $v$ has an ascent at $i$.
\end{lemma}
\begin{proof}
Suppose that $v$ has a descent at $i$.  Then $\pi_i(v) <v$. Applying \cref{cor:pi_iorderpreserving} to $\pi_i(A)\le v$, we obtain $\pi_i(A) = \pi_i(\pi_i(A)) \leq \pi_i(v)$, and so $v \notin \perm(\pi_i(A))$, a contradiction.
\end{proof}

\begin{lemma}\label{lem:reversePiPermSet}
    If $\pi_i(A) \leq v$ and $v$ has an ascent at $i$, then $A \leq vs_i $.
\end{lemma}

\begin{proof}
    Write $A=u_1 \vee  \cdots \vee  u_k$ for bigrassmannian permutations $u_j$, which we know we can do by \cref{lemma:asmbigrass}.  By \Cref{prop:joindescentpart3}, $\pi_i(A) = \pi_i(u_1) \vee \cdots \vee \pi_i(u_k)$.   Thus the assumption $\pi_i(A) \leq v$ implies $\pi_i(u_j) \leq v$ for all $j \in [k]$.  By the definition of join, $A \leq vs_i$ if and only if $u_j \leq vs_i$ for all $j \in [k]$, and so it suffices to show the latter.

    Fix $j \in[k]$.  Because $v$ has an ascent at $i$, $v<vs_i$.  Thus, if $u_j$ has an ascent at $i$, then $u_j = \pi_i(u_j) \leq v<vs_i$.    If $u_j$ has a descent at $i$, then we have $u_js_i = \pi_i(u_j)\leq v<vs_i$.  
Since $vs_i$ has a descent at $i$ and $u_js_i$ has an ascent at $i$, $u_j \leq v s_i$ by the Lifting Property \cite[Proposition 2.2.7]{BB05}.
\end{proof}

\begin{proposition}
\label{prop:intersectoperator}
    Let $A \in \asm(n)$.  Write $I_A = I_{w_1} \cap \cdots \cap I_{w_r}$ with $w_j \in S_n$. Then $I_{\pi_i(A)} = I_{\pi_i(w_1)} \cap \cdots \cap I_{\pi_i(w_r)}$.      In particular, $\perm(\pi_i(A))$ consists of the minimal elements of $\{\pi_i(w) : w \in \perm(A)\}$.
\end{proposition}
\begin{proof}
    Because $I_A \subseteq I_{w_i}$ for each $i \in [r]$, $I_{\pi_i(A)} \subseteq I_{\pi_i(w_1)} \cap \cdots \cap I_{\pi_i(w_r)}$ by \cref{cor:pi_iorderpreserving}.  For the other containment, because $I_{\pi_i(A)} = \bigcap_{w \in \perm(\pi_i(A))} I_w$, it suffices to show $\perm(\pi_i(A)) \subseteq \{\pi_i(u) : u \in \perm(A)\}$.  Fix $v \in \perm(\pi_i(A))$, in which case $\pi_i(A) \leq v$.  By \cref{lem:permpiAascent}, $v$ has an ascent at $i$, and so $A \leq vs_i$ by \cref{lem:reversePiPermSet}.  Hence, there exists some $u \in \perm(A)$ satisfying $A \leq u \leq vs_i$.
    
    Then $\pi_i(A) \leq \pi_i(u) \leq \pi_i(vs_i) = v$ by \cref{cor:pi_iorderpreserving}.  By the assumption $v \in \perm(\pi_i(A))$, $\pi_i(u) = v$, and so $v \in \{\pi_i(u) : u \in \perm(A)\}$.

     The final sentence follows from the previous together with \cref{prop:I_A-intersection-of-schubs-in-perm}.
\end{proof}

\begin{remark}
We note that a statement more general than \cref{prop:intersectoperator} also holds.  For $A \in \asm(n)$, if $I_A = \bigcap_{j \in [t]} I_{B_j}$, then $I_{\pi_i(A)} = \bigcap_{j \in [t]} I_{\pi_i(B_j)}$.  By expanding each $I_{B_j}$ as an intersection of Schubert determinantal ideals, this equality is an easy corollary of \cref{prop:intersectoperator}.  

 Returning to the setting of \cref{meetCounterexample} with $A = 2341 \wedge 3124$, we can see that \cref{prop:intersectoperator} does not extend to the setting of arbitrary meets.  Specifically, $I_{\pi_2(A)} = I_{2134} \neq I_{2341} \cap I_{3124} = I_{\pi_2(2341)} \cap I_{\pi_2(3124)}$.  In this sense, we understand \cref{prop:intersectoperator} to be an intrinsically algebro-geometric proposition about ASM varieties in their capacity as unions of matrix Schubert varieties and not a statement about meets in the lattice $\asm(n)$.
\end{remark}

\begin{remark}
    It is not in general true that $\perm(\pi_i(A)) = \{\pi_i(u) : u \in \perm(A)\}$ because some $I_{\pi_i(w_j)}$ may be redundant in the intersection even if $I_{w_j}$ is not.  For example, if \[
    A = \begin{pmatrix} 0 & 1 & 0\\
    1 & -1 & 1 \\
    0 & 1 & 0
    \end{pmatrix},
    \] then $I_A = I_{231} \cap I_{312}$.  Then $\pi_2(A) = 213 = \pi_2(231)$, and $\pi_2(312) = 312$.  Then $I_{\pi_2(A)} = I_{\pi_2(231)} \cap I_{\pi_2(312)} = I_{\pi_2(231)}$, so $\pi_2(312) \notin \perm(\pi_2(A))$.
\end{remark}

\begin{corollary}\label{cor:perm-set-elements-with-descent-stay-in-perm}
    If $w \in \perm(A)$ and $i$ is a descent of $w$, then $\pi_i(w) \in \perm(\pi_i(A))$.
\end{corollary}
\begin{proof}
    Write $\perm(A) = \{w_1, \ldots, w_r\}$, and assume $w = w_r$.  By \cref{prop:intersectoperator}, $\pi_i(w_r) \in \perm(\pi_i(A))$ unless there exists $j \in [r-1]$ so that $\pi_i(w_j)<\pi_i(w_r)$.  Fix $j \in [r-1]$, and suppose for contradiction that $\pi_i(w_j)<\pi_i(w_r)$.  If $i$ is not a descent of $w_j$, then $w_j = \pi_i(w_j)<\pi_i(w_r)<w_r$.  If $w_j$ does have a descent at $i$, then $w_js_i = \pi_i(w_j) < \pi_i(w_r) = w_rs_i$.  Because both $w_j$ and $w_r$ have descents at $i$, both $w_js_i$ and $w_rs_i$ have ascents at $i$.  Hence $w_js_i < w_rs_i$ implies $w_j<w_r$.  In both cases, then, we have concluded that $w_j<w_r$, which contradicts the assumption $w_j, w_r \in \perm(A)$ because distinct elements of $\perm(A)$ must be incomparable.  
\end{proof}

\begin{corollary}
\label{cor:stepdowncodimbyatmost1}
Let $A\in \asm(n)$ and $i \in [n-1]$.  If $i$ is a descent of $w$ for some $w \in \perm(A)$ with $\ell(w) = \codim(X_A)$, then $\codim (X_{\pi_i(A)}) = \codim(X_A)-1$.  Otherwise, $\codim (X_{\pi_i(A)}) = \codim(X_A)$.  In particular, $\codim(X_A)-\codim(X_{\pi_i(A)}) \in \{0,1\}$.
\end{corollary}

\begin{proof}
For any $w\in S_n$, $\ell(w)-\ell(\pi_i(w))\in\{0,1\}$. By \cref{prop:intersectoperator}, $\perm(\pi_i(A))$ consists of the minimal elements of $\{\pi_i(w) : w \in \perm(A)\}$.  Then by \cref{prop:I_A-intersection-of-schubs-in-perm}, $\codim(X_{\pi_i(A)}) = \min\{\ell(\pi_i(w)) :  w \in \perm(A)\}$.  Thus $\codim(X_{\pi_i(A)}) \geq \codim(X_A)-1$.
 
 Suppose that there exists some $u \in \perm(A)$ with $\ell(u) = \codim(X_A)$ such that $i$ is a descent of $u$.  Then $\codim(X_{\pi_i(A)}) \leq \ell(\pi_i(u)) = \ell(u)-1 = \codim(X_A)-1$.  Hence $\codim(X_{\pi_i(A)})= \codim(X_A)-1$.

 Alternatively, suppose that $i$ is an ascent of $u$ for all $u \in \perm(A)$ satisfying $\ell(u) = \codim(X_A)$.  Then $\ell(\pi_i(u)) = \ell(u) = \codim(X_A)$ for all such $u$, and so $\codim(X_{\pi_i(A)}) \leq \codim(X_A)$.  If $v \in \perm(A)$ so that $\ell(v) \neq \codim(X_A)$, then $\ell(v)>\codim(X_A)$, and so $\ell(\pi_i(v)) \geq \codim(X_A)$.  Hence, $\min\{\ell(\pi_i(w)) :  w \in \perm(A)\} \geq \codim(X_A)$.  Thus $\codim(X_{\pi_i(A)}) = \codim(X_A)$.
\end{proof}

\begin{corollary}\label{cor:perm-set-of-pi_i(A)-when-codim-drops}  Let $A \in \asm(n)$ and $i \in [n-1]$.  If $\codim(X_{\pi_i(A)}) = \codim(X_A)-1$, then 
    \begin{align*}\{u\in \perm(\pi_i(A))&: \ell(u)=\codim(X_{\pi_i(A)})\}\\
    &= \{ws_i:w\in \perm(A),  \ell(w)=\codim(X_A), w>ws_i \}.\end{align*}
\end{corollary}
\begin{proof}
    Combine \cref{prop:I_A-intersection-of-schubs-in-perm},  \cref{prop:intersectoperator}, and \cref{cor:stepdowncodimbyatmost1}.  
\end{proof}

\begin{corollary}\label{cor:descent-of-A-implies-descent-in-perm}
    Let $A \in \asm(n)$ and $i \in [n-1]$.  Then $i$ is a descent of $A$ if and only if there exists $w \in \perm(A)$ so that $i$ is a descent of $w$.
\end{corollary}
\begin{proof}
    Suppose that $i$ is an ascent of $w$ for all $w \in \perm(A)$.  Then by \cref{prop:I_A-intersection-of-schubs-in-perm} and \cref{prop:intersectoperator}, \[
    I_{\pi_i(A)} = \bigcap_{w \in \perm(A)} I_{\pi_i(w)} = \bigcap_{w \in \perm(A)} I_w = I_A,
    \] which is to say that $i$ is an ascent of $A$.

    Conversely, suppose that $i$ is a descent of $w$ for some $w \in \perm(A)$.  By \cref{cor:perm-set-elements-with-descent-stay-in-perm}, $\pi_i(w) \in \perm(\pi_i(A))$.  Because $\pi_i(w)<w \in \perm(A)$ and elements of $\perm(A)$ are incomparable, $\pi_i(w) \notin \perm(A)$.  Hence, $A \neq \pi_i(A)$, and so $i$ is a descent of $A$.
\end{proof}

\begin{corollary}\label{cor:some-pi-drops-codim}
    Let $A\in \asm(n)$, and assume that $A$ is not the identity element of $S_n$.  Then there exists $i\in [n-1]$ so that $\codim(X_{\pi_i(A)}) = \codim(X_A) -1$.
\end{corollary}
\begin{proof}
Because $A$ is not the identity element of $S_n$, there exists some non-identity $w \in \perm(A)$ satisfying $\ell(w) = \codim(X_A)$.  Because $w$ is not the identity, $w$ has at least one descent $i \in [n-1]$. By  \cref{cor:stepdowncodimbyatmost1}, $\codim(X_{\pi_i(A)}) = \codim(X_A) -1$.
\end{proof}

\subsection{Codimension and components via weak order chains}

Let $A \in \asm(n)$, and let $e$ denote the identity element of $S_n$.  Given a saturated weak order chain $e=A_0\prec A_1 \prec \cdots \prec A_{k}=A$, we may encode the chain by the tuple $(a_1,\ldots,a_k) \in [n-1]^k$ where $A_{i-1}=\pi_{a_{i}}(A_{i})$ for all $i\in [k]$.
Let \newword{chains$(A)$} denote the set of all such tuples representing saturated chains from the identity to $A$ in weak order.

\begin{theorem}
\label{cor:codimension-saturated-chains}
(1)    Let $A \in \asm(n)$. Then $(a_1, \ldots, a_k) \in \chains(A)$ if and only if a substring of $(a_1, \ldots, a_k)$ forms a reduced word for some $w \in \perm(A)$, $\pi_{a_k}(A) \neq A$, and no $\pi_{a_j}$ acts trivially on $\pi_{a_{j+1}} \circ \cdots \circ \pi_{a_k}(A)$ for $j \in [k-1]$.

(2) If $(a_1, \ldots, a_k)$ forms a reduced word for some $w \in \perm(A)$, then $(a_1, \ldots, a_k) \in \chains(A)$.  In particular, \[
    \perm(A)\subseteq \{s_{a_1}s_{a_2}\cdots s_{a_k}:(a_1,\ldots,a_k)\in\chains(A)\}.
    \] 
    
(3) $\codim(X_A)$ is equal to the minimum length of a saturated chain in the weak order poset from the identity permutation $e$ to $A$.
\end{theorem}
\begin{proof}
(1) Fix a string $s = (s_{a_1}, \ldots, s_{a_k})$ of simple transpositions, and let $\pi = \pi_{a_1} \circ \cdots \circ \pi_{a_k}$.  Then a substring of $s$ represents a chain from $e$ to $A$ in the weak order poset if and only if $\pi(A) = e$, $\pi_k(A) \neq A$, and no $a_j$ acts trivially on $\pi_{a_{j+1}} \circ \cdots \circ \pi_{a_k}(A)$ for $j \in [k-1]$.  Note that $\pi(A) = e$ if and only if $I_{\pi(A)} = (0)$.

Write $I_A = \bigcap_{w \in \perm(A)} I_w$ by \cref{prop:I_A-intersection-of-schubs-in-perm}.  By \cref{prop:intersectoperator}, $I_{\pi(A)} = \bigcap_{w \in \perm(A)} I_{\pi(w)}$.  Hence, $I_{\pi(A)} = 0$ if and only if $I_{\pi(w)} = 0$ for some $w \in \perm(A)$ if and only if $\pi(w) = e$ for some $w \in \perm(A)$. For $w \in S_n$, we claim that $\pi(w) = e$ if and only if $s = (s_{a_1}, \ldots, s_{a_k})$ contains as a substring a reduced expression for $w$.  By considering the substring of $s$ whose corresponding operator $\pi_{a_j}$ acts nontrivially on $\pi_{a_{j+1}} \circ \cdots \circ \pi_{a_k}(w)$ (including $s_{a_k}$ if and only if $\pi_{a_k}(w) \neq w$), the result follows from the equality $\pi_i(w) = ws_i$ whenever $\pi_i$ acts nontrivially on $w$.

(2) If $(a_1, \ldots, a_k)$ itself forms a reduced word for some $w \in \perm(A)$, we claim that $(a_1, \ldots, a_k) \in \chains(A)$.  From Part (1), it suffices to show that $a_k$ is a descent of $A$ and that $a_j$ is a descent of $\pi_{j+1} \circ \cdots \circ \pi_k(A)$ for all $j \in [k-1]$.  This follows from \cref{cor:perm-set-elements-with-descent-stay-in-perm} and \cref{cor:descent-of-A-implies-descent-in-perm}.

(3) This follows from \cref{prop:I_A-intersection-of-schubs-in-perm} together with Parts (1) and (2).
\end{proof}

\begin{example}
    Let $A=\begin{pmatrix} 
    0 & 0 & 1 & 0\\
    1 & 0  & -1 & 1 \\
    0 & 1 & 0 & 0\\
    0 & 0 & 1 & 0
    \end{pmatrix}$. 
    We have that $\perm(A)=\{4123,3412\}$.  Furthermore, $\ell(4123)=3$ and $\ell(3412)=4$.  Thus, $\codim(X_A)=3$ by \cref{prop:I_A-intersection-of-schubs-in-perm}.  We can also see $\codim(X_A)=3$ using weak order chains.  Pictured below is the order ideal of elements below $A$ in weak order.
	\begin{center}
	\scalebox{0.72}
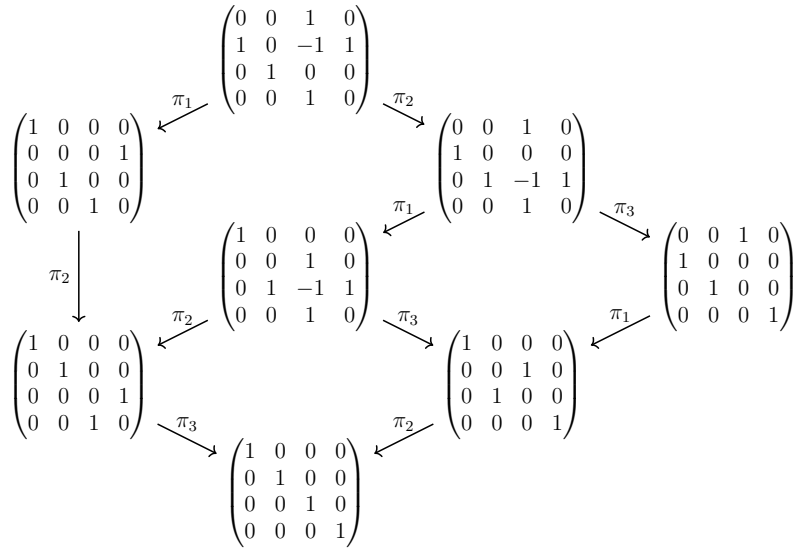
{
\begin{tikzpicture}
	\node (A) at (0,0) {
    $\begin{pmatrix} 
    0 & 0 & 1 & 0\\
    1 & 0  & -1 & 1 \\
    0 & 1 & 0 & 0\\
    0 & 0 & 1 & 0
    \end{pmatrix}$};
	\node (B) at (-4,-2) {
 $\begin{pmatrix} 
    1 & 0 & 0 & 0\\
    0 & 0  & 0 & 1 \\
    0 & 1 & 0 & 0\\
    0 & 0 & 1 & 0
    \end{pmatrix}$
    };
    \node (C) at (4,-2) {
 $\begin{pmatrix} 
    0 & 0 & 1 & 0\\
    1 & 0  & 0 & 0 \\
    0 & 1 & -1 & 1\\
    0 & 0 & 1 & 0
    \end{pmatrix}$
    };
	\node (D) at (0,-4) {
 $\begin{pmatrix} 
    1 & 0 & 0 & 0\\
    0 & 0  & 1 & 0 \\
    0 & 1 & -1 & 1\\
    0 & 0 & 1 & 0
    \end{pmatrix}$
    };
 \node (E) at (8,-4) {
 $\begin{pmatrix} 
    0 & 0 & 1 & 0\\
    1 & 0  & 0 & 0 \\
    0 & 1 & 0 & 0\\
    0 & 0 & 0 & 1
    \end{pmatrix}$
 };
 \node (F) at (-4,-6) {
 $\begin{pmatrix} 
    1 & 0 & 0 & 0\\
    0 & 1  & 0 & 0 \\
    0 & 0 & 0 & 1\\
    0 & 0 & 1 & 0
    \end{pmatrix}$
 };
 \node (G) at (4,-6) {
 $\begin{pmatrix} 
    1 & 0 & 0 & 0\\
    0 & 0  & 1 & 0 \\
    0 & 1 & 0 & 0\\
    0 & 0 & 0 & 1
    \end{pmatrix}$
 };
 \node (H) at (0,-8) {
 $\begin{pmatrix} 
    1 & 0 & 0 & 0\\
    0 & 1  & 0 & 0 \\
    0 & 0 & 1 & 0\\
    0 & 0 & 0 & 1
    \end{pmatrix}$
 };
	\draw[->,thick] (A) -- node[above]{$\pi_1$} (B) ;
    \draw[->,thick] (A) -- node[above]{$\pi_2$} (C) ;
    \draw[->,thick] (B) -- node[left]{$\pi_2$} (F) ;
    \draw[->,thick] (C) -- node[above]{$\pi_1$} (D) ;
    \draw[->,thick] (C) -- node[above]{$\pi_3$} (E) ;
    \draw[->,thick] (D) -- node[above]{$\pi_2$} (F) ;
    \draw[->,thick] (D) -- node[above]{$\pi_3$} (G) ;
    \draw[->,thick] (E) -- node[above]{$\pi_1$} (G) ;
    \draw[->,thick] (F) -- node[above]{$\pi_3$} (H) ;
    \draw[->,thick] (G) -- node[above]{$\pi_2$} (H) ;
\end{tikzpicture}}
	\end{center}
	The shortest element of $\chains(A)$ is $(3,2,1)$, which is of length $3$.  We also have $3=\codim(X_A)$, as predicted by \cref{cor:codimension-saturated-chains}.  
    
    We see in this example that there may be elements of $\chains(A)$ that are not reduced words for any elements of $\perm(A)$.  For example, $(3,2,1,2) \in \chains(A)$, but $s_3s_2s_1s_2 = 4213 \notin \perm(A)$.  However, $(3,2,1)$ is also an element of $\chains(A)$, and $s_3s_2s_1 = 4123 \in \perm(A)$.  The issue is that $\pi_2 \circ \pi_1 \circ \pi_2(A) = \pi_2 \circ \pi_1(A)$, a possibly unintuitive relation.  We use the next proposition to give some information about the chains in the weak order poset that do not come from elements of $\perm(A)$.
\end{example}

Given a reduced word $(a_1,\ldots,a_k)$ for $w$, define $\pi_w=\pi_{a_1}\circ \cdots \circ \pi_{a_k}$.  By \cite[Proposition 3.2]{HR20}, $\pi_w$ is independent of the choice of reduced word.


\begin{proposition}\label{prop:reduced-words-and-saturated-chains}
    Let $A\in \asm(n)$.   
    \begin{enumerate}
    \item If $(a_1,\ldots,a_k)\in\chains(A)$, then it is a reduced word. 
        \item If $w=s_{a_1}s_{a_2}\cdots s_{a_k}$ for some $(a_1,\ldots,a_k)\in \chains(A)$ then $w\geq v$ for some $v\in \perm(A)$.
    \end{enumerate}
\end{proposition}
\begin{proof}

\noindent (1) Write $A_k = A$ and $A_{i-1} = \pi_{a_i}(A_i)$ for each $i \in [k]$.  If $(a_1,\ldots,a_k)$ is not a reduced word, then there exists some index $j$ so that 
\[s_{a_j} s_{a_{j+1}} \cdots s_{a_k}<s_{a_{j+1}} s_{a_{j+2}} \cdots s_{a_k}.\]
Suppose that $j$ has been chosen maximally, in which case $s_{a_{j+1}} s_{a_{j+2}} \cdots s_{a_k}$ is a reduced expression. Write $v = s_{a_{j+1}} s_{a_{j+2}} \cdots s_{a_k}$.  The inequality $s_{a_j}v<v$ implies that we may also write $v = s_{a_j}s_{a'_{j+2}}\cdots s_{a'_k}$ for some reduced word $(a_j,a'_{j+2}, \ldots, a'_k)$.  Write $A'_k = A$ and $A'_{i-1} = \pi_{a'_i}(A'_i)$ for $i \in [j+2,k]$.  Then $A_j = \pi_v(A) = \pi_{a_j}(A'_{j+1})$.  But the action of any $\pi_i$ is idempotent, and so $\pi_{a_j}(A_j) = \pi_{a_j}(\pi_{a_j}(A'_{j+1})) = \pi_{a_j}(A'_{j+1}) = A_j$, in violation of the definition of chain.

\noindent (2) By Part (1), $(a_1, \ldots, a_k)$ is a reduced word for $w$.  By \cref{cor:codimension-saturated-chains}, some substring of $(a_1, \ldots, a_k)$ is a reduced word for some $v \in \perm(A)$.  Hence $w \geq v$.
\end{proof}

\begin{proposition}\label{prop:saturated-chains-same-length}
    Let $A\in \asm(n)$.  The weak order saturated chains from the identity to $A$ are all of the same length if and only if $X_B$ is equidimensional for all $B\preceq A$.
\end{proposition}
\begin{proof}

\noindent ($\Rightarrow$) We will prove the contrapositive.  Suppose there exists some $B\preceq A$ so that $X_B$ is not equidimensional.  We must show that there are saturated chains of different lengths from the identity to $A$.

For each $w\in \perm(B)$, there exists at least one saturated chain of length $\ell(w)$ from the identity to $B$  by \cref{cor:codimension-saturated-chains}.  By \cref{prop:I_A-intersection-of-schubs-in-perm}, $X_B$ is equidimensional if and only if all elements of $\perm(B)$ are of the same length. Hence, if $X_B$ is not equidimensional, then there are saturated chains of different lengths from the identity to $B$. Fix two such chains, $\mathcal{C}_1$ and $\mathcal{C}_2$, of different lengths. 

Because $B\preceq A$, we may choose a saturated chain $\mathcal{C}$ from $B$ to $A$. Extending the two $\mathcal{C}_i$ by $\mathcal{C}$ constructs saturated chains of different lengths from the identity to $A$.

\noindent ($\Leftarrow$) We use induction on $\asm(n)$ under strong order.  The statement is trivially true for the identity.  Now fix $A\in \asm(n)$, and suppose that the statement holds for all $A'<A$.  Assume that $X_B$ is equidimensional for all $B \preceq A$.

Then, in particular, if $\pi_i(A)\prec A$ is a weak order cover, we have $\pi_i(A) < A$ and the inductive hypothesis holds for $\pi_i(A)$.  Each saturated chain from the identity to $A$ is formed by concatenating a saturated chain from the identity to $\pi_i(A)$ for some $i$ in the descent set of $A$ with the chain recording the cover $\pi_i(A) \prec A$.  So it is enough to show that the codimensions of the $X_{\pi_i(A)}$ are constant for all $i$ in the descent set of $A$.  Fix any such $i$.  By \cref{cor:descent-of-A-implies-descent-in-perm}, there exists $w \in \perm(A)$ so that $i$ is a descent of $w$.  Because $X_A$ is equidimensional, $\ell(w) = \codim(X_A)$.  Hence by \cref{cor:stepdowncodimbyatmost1}, $\codim(X_{\pi_i(A)})=\codim(X_A)-1$.  
\end{proof}

The authors consider \cref{prop:saturated-chains-same-length} to be somewhat surprising for the following reason: It is easy to see that $X_A$ is equidimensional if and only if $X_{A^T}$ is.  Indeed, the assignments $z_{i,j} \mapsto z_{j,i}$ induce an isomorphism on the varieties' coordinate rings.  Meanwhile, the weak order intervals between the identity and $A$ and $A^T$, respectively, do not enjoy any such strong, naive connection.  
See \cref{figure:example_of_A} and \cref{figure:example_of_A_transpose} for an example where the weak order intervals for $A$ and $A^T$ are not anti-isomorphic. We thank Vic Reiner for pointing out that, moreover, these intervals have different homotopy types and different M\"obius functions.  (For $w\in S_n$, it is true that the order ideals for $w$ and $w^T$ are anti-isomorphic posets.)

    \begin{figure}[!htb]
    \centering
    \begin{minipage}{.45\textwidth}
        \centering
            \scalebox{0.49}{
        \begin{tikzpicture}
        \node at (0,4.5){$A$};
        \node at (0,3.7){{\rotatebox{90}{$\,=$}}};
	\node (A) at (0,2) {
    $\begin{pmatrix} 
    0 & 0 & 1 & 0\\
    0 & 1  & -1 & 1 \\
    1 & 0 & 0 & 0\\
    0 & 0 & 1 & 0
    \end{pmatrix}$};
	\node (B) at (-4,-2) {
 $\begin{pmatrix} 
    0 & 1 & 0 & 0\\
    0 & 0  & 0 & 1 \\
    1 & 0 & 0 & 0\\
    0 & 0 & 1 & 0
    \end{pmatrix}$
    };
    \node (C) at (4,-2) {
 $\begin{pmatrix} 
    0 & 0 & 1 & 0\\
    1 & 0  & 0 & 0 \\
    0 & 1 & -1 & 1\\
    0 & 0 & 1 & 0
    \end{pmatrix}$
    };
	\node (D) at (2,-6) {
 $\begin{pmatrix} 
    1 & 0 & 0 & 0\\
    0 & 0  & 1 & 0 \\
    0 & 1 & -1 & 1\\
    0 & 0 & 1 & 0
    \end{pmatrix}$
    };
 \node (E) at (7,-6) {
 $\begin{pmatrix} 
    0 & 0 & 1 & 0\\
    1 & 0  & 0 & 0 \\
    0 & 1 & 0 & 0\\
    0 & 0 & 0 & 1
    \end{pmatrix}$
 };
 \node (F) at (0,-10) {
 $\begin{pmatrix} 
    1 & 0 & 0 & 0\\
    0 & 1  & 0 & 0 \\
    0 & 0 & 0 & 1\\
    0 & 0 & 1 & 0
    \end{pmatrix}$
 };
 \node (G) at (4,-10) {
 $\begin{pmatrix} 
    1 & 0 & 0 & 0\\
    0 & 0  & 1 & 0 \\
    0 & 1 & 0 & 0\\
    0 & 0 & 0 & 1
    \end{pmatrix}$
 };
 \node (H) at (0,-14) {
 $\begin{pmatrix} 
    1 & 0 & 0 & 0\\
    0 & 1  & 0 & 0 \\
    0 & 0 & 1 & 0\\
    0 & 0 & 0 & 1
    \end{pmatrix}$
 };
 \node (AA) at (-4,-6) {
  $\begin{pmatrix} 
    0 & 1 & 0 & 0\\
    1 & 0  & 0 & 0 \\
    0 & 0 & 0 & 1\\
    0 & 0 & 1 & 0
    \end{pmatrix}$
 };
  \node (BB) at (-4,-10) {
  $\begin{pmatrix} 
    0 & 1 & 0 & 0\\
    1 & 0  & 0 & 0 \\
    0 & 0 & 1 & 0\\
    0 & 0 & 0 & 1
    \end{pmatrix}$
 };
	\draw[->,thick] (A) -- node[above left ]{$1$} (B) ;
    \draw[->,thick] (A) -- node[above right]{$2$} (C) ;
    \draw[->,thick] (B) -- node[left]{$2$} (AA) ;
    \draw[->,thick] (C) -- node[left]{$1$} (D) ;
    \draw[->,thick] (C) -- node[right]{$3$} (E) ;
    \draw[->,thick] (D) -- node[left]{$2$} (F) ;
    \draw[->,thick] (D) -- node[right]{$3$} (G) ;
    \draw[->,thick] (E) -- node[right]{$1$} (G) ;
    \draw[->,thick] (F) -- node[right]{$3$} (H) ;
    \draw[->,thick] (G) -- node[below right]{$2$} (H) ;
    \draw[->,thick] (AA) -- node[left]{$3$} (BB) ;
    \draw[->,thick] (AA) -- node[right]{$1$} (F) ;
    \draw[->,thick] (BB) -- node[below left]{$1$} (H) ;
\end{tikzpicture}
}
\caption{Weak order interval from an ASM $A$ to the identity. An edge label $i \in [3]$ indicates a covering relation determined by $\pi_i$.  This interval also appears in \cite[Figure 1]{HR20}.}
    \label{figure:example_of_A}
    \end{minipage}%
    \begin{minipage}{0.55\textwidth}
        \centering
	\scalebox{0.49}{
\begin{tikzpicture}
        \node at (0,6.2){$A^T$};
        \node at (0,5.5){{\rotatebox{90}{$\,=$}}};
	\node (A) at (0,4) {
    $\begin{pmatrix} 
    0 & 0 & 1 & 0\\
    0 & 1  & 0 & 0 \\
    1 & -1 & 0 & 1\\
    0 & 1 & 0 & 0
    \end{pmatrix}$};
	\node (B) at (-8,-1) {
 $\begin{pmatrix} 
    0 & 1 & 0 & 0\\
    0 & 0  & 1 & 0 \\
    1 & -1 & 0 & 1\\
    0 & 1 & 0 & 0
    \end{pmatrix}$
    };
    \node (C) at (0,-1) {
 $\begin{pmatrix} 
    0 & 0 & 1 & 0\\
    0 & 1  & 0 & 0 \\
    1 & 0 & 0 & 0\\
    0 & 0 & 0 & 1
    \end{pmatrix}$
    };
    \node (D) at (7,-1) {
 $\begin{pmatrix} 
    0 & 0 & 1 & 0\\
    1 & 0  & 0 & 0 \\
    0 & 0 & 0 & 1\\
    0 & 1 & 0 & 0
    \end{pmatrix}$
    };

\node (E) at (-1,-5) {
 $\begin{pmatrix} 
    0 & 1 & 0 & 0\\
    1 & -1  & 1 & 0 \\
    0 & 0 & 0 & 1\\
    0 & 1 & 0 & 0
    \end{pmatrix}$
    };
     \node (F) at (-8,-8) { 
  $\begin{pmatrix} 
    0 & 1 & 0 & 0\\
    0 & 0  & 1 & 0 \\
    1 & 0 & 0 & 0\\
    0 & 0 & 0 & 1
    \end{pmatrix}$
 };
	\node (G) at (7,-8) {
 $\begin{pmatrix} 
    0 & 0 & 1 & 0\\
    1 & 0  & 0 & 0 \\
    0 & 1 & 0 & 0\\
    0 & 0 & 0 & 1
    \end{pmatrix}$
    };
 \node (H) at (2,-8) {
 $\begin{pmatrix} 
    1 & 0 & 0 & 0\\
    0 & 0  & 1 & 0 \\
    0 & 0 & 0 & 1\\
    0 & 1 & 0 & 0
    \end{pmatrix}$
 };
 \node (K) at (4,-12) {
 $\begin{pmatrix} 
    1 & 0 & 0 & 0\\
    0 & 0  & 1 & 0 \\
    0 & 1 & 0 & 0\\
    0 & 0 & 0 & 1
    \end{pmatrix}$
 };
 \node (L) at (0,-16) { 
 $\begin{pmatrix} 
    1 & 0 & 0 & 0\\
    0 & 1  & 0 & 0 \\
    0 & 0 & 1 & 0\\
    0 & 0 & 0 & 1
    \end{pmatrix}$
 };

  \node (J) at (-4,-12) {
  $\begin{pmatrix} 
    0 & 1 & 0 & 0\\
    1 & 0  & 0 & 0 \\
    0 & 0 & 1 & 0\\
    0 & 0 & 0 & 1
    \end{pmatrix}$
 };
   \node (I) at (-4,-8) {
  $\begin{pmatrix} 
    0 & 1 & 0 & 0\\
    1 & -1  & 1 & 0 \\
    0 & 1 & 0 & 0\\
    0 & 0 & 0 & 1
    \end{pmatrix}$
 };
	\draw[->,thick] (A) -- node[above left]{$1$} (B) ;
    \draw[->,thick] (A) -- node[left]{$3$} (C) ;
    \draw[->,thick] (A) -- node[above]{$2$} (D) ;
    \draw[->,thick] (B) -- node[above right ]{$2$} (E) ;
    \draw[->,thick] (B) -- node[left]{$3$} (F) ;
    \draw[->,thick] (C) -- node[above left ]{$1$} (F) ;
    \draw[->,thick] (C) -- node[below left]{$2$} (G) ;
    \draw[->,thick] (D) -- node[right]{$3$} (G) ;
    \draw[->,thick] (D) -- node[right]{$1$} (H) ;
    \draw[->,thick] (G) -- node[below right]{$1$} (K) ;
    \draw[->,thick] (H) -- node[right]{$3$} (K) ;
    \draw[->,thick] (K) -- node[below right]{$2$} (L) ;
    \draw[->,thick] (J) -- node[below left]{$1$} (L) ;
    \draw[->,thick] (F) -- node[below left]{$2$} (J) ;
    \draw[->,thick] (I) -- node[above right]{$2$} (J) ;
    \draw[->,thick] (E) -- node[above right]{$1$} (H) ;
    \draw[->,thick] (I) -- node[above right]{$1$} (K) ;
    \draw[->,thick] (E) -- node[above left ]{$3$} (I) ;
\end{tikzpicture}}
    \caption{Weak order interval from the ASM $A^T$ to the identity where $A$ is the ASM from \cref{figure:example_of_A}. An edge label $i \in [3]$ indicates a covering relation determined by $\pi_i$.}
    \label{figure:example_of_A_transpose}
    \end{minipage}
\end{figure}

\begin{conjecture}\label{conj:characterize-CM-by-chains}
    Let $A\in \asm(n)$.  The weak order saturated chains from the identity to $A$ are all of the same length if and only if $X_B$ is Cohen--Macaulay for all $B\preceq A$.
\end{conjecture}

\begin{remark} Because Cohen--Macaulay varieties are necessarily equidimensional, the backward direction of \cref{conj:characterize-CM-by-chains} follows from \cref{prop:saturated-chains-same-length}.  As argued in the proof of \cref{prop:saturated-chains-same-length}, the property of having all weak order chains from the identity of the same length is preserved under application of any $\pi_i$.  Hence, in order to prove \cref{conj:characterize-CM-by-chains}, it would be sufficient to show the following: If the weak order chains from the identity to $A$ are all of the same length, then $X_A$ is Cohen--Macaulay.  \cref{conj:characterize-CM-by-chains} holds for all $A \in \asm(n)$ for $n \leq 6$.
\end{remark}

\begin{example}
    In order to gain intuition for conditions of the form \say{for all $B \preceq A$} appearing in the above proposition and conjecture, we note that $X_{\pi_i(A)}$ may satisfy fewer desirable algebraic properties than $X_A$ does.  For example, consider \[
    A = \begin{pmatrix}
        0 & 0 & 1 & 0 \\
        0 & 1 & 0 & 0\\
        1 & -1 & 0 & 1\\
        0 & 1 & 0 & 0
    \end{pmatrix} \quad \mbox{ and } \quad \pi_1(A) = \begin{pmatrix}
        0 & 1 & 0 & 0 \\
        0 & 0 & 1 & 0\\
        1 & -1 & 0 & 1\\
        0 & 1 & 0 & 0
    \end{pmatrix}.
    \] Then $I_A = (z_{1,1}, z_{1,2}, z_{2,1}, z_{2,2}z_{3,1})$ while $I_{\pi_1(A)} = (z_{1,1}, z_{2,1}, z_{2,2}z_{3,1}, z_{1,2}z_{3,1})$.  Observe that $X_A$ is a complete intersection, hence equidimensional, while $X_{\pi_1(A)}$ has one component of codimension $3$ and another of codimension $4$.
\end{example}

Very little is currently known about Cohen--Macaulayness of ASM varieties.  It is not even known whether or not $X_A$ is Cohen--Macaulay if and only if $X_{[1] \oplus A}$ is Cohen--Macaulay, where $[1] \oplus A$ denotes the block sum.  This equivalence, conjectured in \cite{AFH+}, is one example of a statement that would be readily deducible were \cref{conj:characterize-CM-by-chains} proved.  

\section{$K$-polynomials and Grothendieck polynomials}

Our next goal is to study Grothendieck polynomials - otherwise known as twisted $K$-polynomials - of ASM varieties.  In order to do so, we will first introduce the relevant concepts and notation, following the reference \cite[Chapter 8]{MS05}.

\subsection{Preliminaries}
Let $R=\field[z_1,\ldots,z_m]$ be a polynomial ring, and fix an $\mathbb{N}^d$-grading. 
Let 
\begin{equation*}
    M=\bigoplus_{\mathbf a\in\mathbb{N}^d}M_{\mathbf{a}}
\end{equation*}
be a finitely generated $\mathbb{N}^d$-graded $R$-module, and assume that $\dim_\field(M_{\mathbf{a}})<\infty$ for all $\mathbf a\in\mathbb{N}^d$.
The \newword{multigraded Hilbert series} of $M$ is the formal power series in the variables $t_1,\ldots,t_d$ given by
\begin{equation*}
    \hilb(M;\mathbf t)=\sum_{\mathbf{a}\in\mathbb{N}^d}\dim_\field(M_{\mathbf{a}})\mathbf{t}^{\mathbf{a}},
\end{equation*}
where $\mathbf{t}^{\mathbf{a}}=t_1^{a_1}\cdots t_d^{a_d}$.  There exists a unique polynomial $\mathcal{K}(M;\mathbf{t})\in\Z[t_1,\ldots,t_d]$, called the \newword{$K$-polynomial} of $M$, such that 
\begin{equation*}
    \hilb(M;\mathbf{t})=\frac{\mathcal{K}(M;\mathbf{t})}{\prod_{i\in[m]}(1-\mathbf{t}^{\deg(z_i)})}.
\end{equation*}

Given a homogeneous (with respect to the relevant grading) ideal $J$ of $R$ and the $K$-polynomial $\mathcal K(R/J;\mathbf t)=\mathcal K(R/J;t_1,\ldots,t_d)$, we define \[\widetilde{\mathcal K}(R/J;\mathbf t)=\mathcal K(R/J;1-\mathbf t)=\mathcal K(R/J;1-t_1,\ldots,1-t_d)\] to be the \newword{twisted $K$-polynomial} of $R/J$.  This twisting has geometric interpretation in terms of the Poincar\'e isomorphism, the result of which is that twisted $K$-polynomials represent classes in $K$-homology (while untwisted $K$-polynomials represent classes in $K$-cohomology).  For a description of this correspondence, see \cite[Remark 2.3.5]{KM05} and \cite[Chapter 15]{Ful98}.  We prefer to work with twisted $K$-polynomials primarily for combinatorial reasons, as borne out below and as has been done in, for example, \cite{EY17, MSD20, BS22, DMS24, PS24, PY24, NS25}, among many others.  

In particular, twisted K-polynomials of matrix Schubert varieties have a cancellation-free combinatorial formula in terms of pipe dreams (\cite{FK94FPSAC, BB93, KM04, KM05}) which aid in the analysis of these polynomials.  We will discuss these polynomials and their generalizations in \cref{subsection:schubertandgrothendieck}.

The \newword{multidegree} of $R/J$, denoted $\mathcal C(R/J;\mathbf t)$, is the sum of the lowest degree terms of $\widetilde{\mathcal K}(R/J;\mathbf t)$.  

Recall that if $\sigma$ is a term order and $J$ a is homogeneous ideal, then we have the equality $\hilb(R/J) = \hilb(R/\init_\sigma(J))$, from which it follows also that $\widetilde{\mathcal K}(R/J;\mathbf t) = \widetilde{\mathcal K}(R/\init_\sigma(J);\mathbf t)$ and $\mathcal C(R/J;\mathbf t) = \mathcal C(R/\init_\sigma(J);\mathbf t)$.  In order to understand these polynomials inductively in the case of ASM ideals, it will be valuable for us to use the recursions described below, which require us to work briefly with certain initial ideals of ASM ideals rather than with ASM ideals directly.

\begin{prop}\label{prop:antiDiagInit}
    Let $I_{a_1}, \ldots, I_{a_k}$ be ASM ideals.  Set $J = \bigcap_{i \in [k]} I_{a_i}$.  Let $\sigma$ be an antidiagonal term order. Then $\init_{\sigma}(J) = \bigcap_{i \in [k]} \init_\sigma(I_{a_i})$.
\end{prop}
\begin{proof}
    See \cite[Section 7.2]{Knu09}.
\end{proof}

\begin{lemma}
\label{lemma:hilbertaltsumprimer}
    Let $R=\field[z_1,\ldots,z_m]$ be a polynomial ring equipped with an $\mathbb{N}^d$-grading.  Let $I_1,\ldots, I_k$ be monomial ideals of $R$.  If $J=\bigcap_{i\in[k]}I_i$ then
    \[\hilb(R/J;\mathbf t)=\sum_{\emptyset \neq U\subseteq [k]}(-1)^{|U|-1}\hilb\left(R/\sum_{i\in U}I_i;\mathbf t\right).\]
\end{lemma}
\cref{lemma:hilbertaltsumprimer} is an exercise.  

\begin{lemma}
\label{lemma:hilbertaltsum}
    Let $I_{A_1},\ldots, I_{A_k}$ be ASM ideals.  If $J=\bigcap_{i\in[k]}I_{A_i}$, then \[
    \hilb(R/J;\mathbf t)=\sum_{\emptyset \neq U\subseteq [k]}(-1)^{|U|-1}\hilb\left(R/\sum_{i\in U}I_{A_i};\mathbf t\right).
    \]
\end{lemma}
\begin{proof}
    Fix an antidiagonal term order $\sigma$. Combining \cref{prop:antiDiagInit} and \cref{lemma:hilbertaltsumprimer}, we have \begin{align*}
\hilb(R/J;\mathbf t) &= \hilb(R/\init_\sigma(J);\mathbf t) \\
&= \sum_{\emptyset \neq U\subseteq [k]}(-1)^{|U|-1}\hilb\left(R/\sum_{i\in U}\init_\sigma(I_{A_i});\mathbf t\right) \\
&= \sum_{\emptyset \neq U\subseteq [k]}(-1)^{|U|-1}\hilb\left(R/\sum_{i\in U}I_{A_i};\mathbf t\right). \qedhere
    \end{align*}
\end{proof}

As an immediate consequence, we have the following:
\begin{corollary}
\label{cor:kandtwistk}
    Let $I_{A_1},\ldots, I_{A_k}$ be ASM ideals.  If $J=\bigcap_{i\in[k]}I_{A_i}$, then
    \[\mathcal K(R/J;\mathbf t)=\sum_{\emptyset \neq U\subseteq [k]}(-1)^{|U|-1}\mathcal K\left(R/\sum_{i\in U}I_{A_i};\mathbf t\right)\] and
     \[\widetilde{\mathcal K}(R/J;\mathbf t)=\sum_{\emptyset \neq U\subseteq [k]}(-1)^{|U|-1} \widetilde{\mathcal K}\left(R/\sum_{i\in U}I_{A_i};\mathbf t\right).\]
\end{corollary}

\subsection{Schubert and Grothendieck polynomials of ASM varieties}
\label{subsection:schubertandgrothendieck}

We now specialize $K$-polynomials to the modules and gradings we are interested in in this paper.
Given $S=\field [z_{i,j} : i,j\in[n]]$ we consider the $\mathbb{N}^{2n}$-grading given by 
\begin{equation*}
    \deg(z_{i,j})=e_i+e_{n+j}.
\end{equation*}
Note then that $\mathbf{t}^{\deg(z_{i,j})}=t_it_{n+j}$.
It is worthwhile to write these $K$-polynomials as elements of $\Z[x_1,\ldots,x_n,y_1,\ldots,y_n]$, where the variables are given by 
\begin{equation*}
    x_1=t_1,\ldots,x_n=t_n \quad\text{and}\quad y_1=t_{n+1},\ldots,y_n=t_{2n}.
\end{equation*}
Under this convention we have that $(\mathbf{x},\mathbf{y})^{\deg(z_{i,j})}=x_iy_j$.

Given $A\in \asm(n)$ we define $\mathfrak G_A(\mathbf x,\mathbf y)=\tilde{\mathcal K}(S/I_A;\mathbf x,\mathbf y)$ and call it the \newword{double ASM Grothendieck polynomial}.  We write $\mathfrak G_A(\mathbf x)=\mathfrak G_A(\mathbf x,\mathbf 0)$ for the \newword{(single) ASM Grothendieck polynomial} i.e., the result of specializing the $y_i$'s to $0$.  Likewise, we define $\mathfrak S_A(\mathbf x,\mathbf y)=\mathcal C(S/I_A;\mathbf x,-\mathbf y)$ to be the \newword{double ASM Schubert polynomial} and $\mathfrak S_A(\mathbf x)=\mathfrak S_A(\mathbf x,\mathbf 0)$ to be the \newword{single ASM Schubert polynomial}.  Note that, equivalently, $\mathfrak G_A(\mathbf x)$ and $\mathfrak S_A(\mathbf x)$ are the twisted $K$-polynomial and multidegree, respectively, of $S/I_A$ with respect to the grading $\deg(z_{i,j})=e_i$, again letting $t_i=x_i$ for $i\in [n]$.

We discussed the role of Schubert polynomials in intersection theory in \cref{sec:intro}, where we recalled that expanding products of Schubert polynomials as sums of other Schubert polynomials records complete information about the cohomology of the complete flag variety. From this perspective, it is clear that distinct permutations give rise to distinct Schubert polynomials.  Correspondingly, distinct matrix Schubert varieties give rise to distinct Schubert polynomials in their capacity as multidegrees (see \cite[Section 3.2]{KM05}).

By contrast, ASM Schubert polynomials cannot reliably see enough of an ASM variety to be as complete a tool.  Specifically, because multidegrees are additive over top-dimensional components of a variety \cite[Theorem 1.7.1]{KM05}, they necessarily fail to record any information about components of other dimensions.  \cref{ex:schubert-forgets-small-component} includes an example of two ASM varieties (one of which is a component of the other), which share the same ASM Schubert polynomial.   
It is to accommodate this type of failure of sensitivity of cohomology that we are motivated to study the K-theory classes of ASM varieties, specifically their ASM Grothendieck polynomials.  One goal of this section is to show that distinct ASMs have distinct ASM Grothendieck polynomials (\cref{prop:equal_grothendiecks_means_equal_ASMs}).  More generally, we aim to build the case that ASM Grothendieck polynomials are often a more suitable tool for studying ASM varieties than ASM Schubert polynomials are, even in settings when Schubert polynomials are an excellent tool for studying matrix Schubert varieties.  In order to do this, we will now study divided difference operators.

\subsection{ASM varieties via Schubert and Grothendieck polynomials}
\label{sec:divided-differences}

For permutations, there is an equivalent definition of Grothendieck polynomials in terms of divided difference operators.  In this subsection, we will show that these divided difference relations also hold for ASM Grothendieck polynomials.  The definition of Grothendieck polynomials in terms of divided differences is important in its capacity as the original method for producing explicit representatives of classes in the K-theory of Schubert varieties.  It is also a valuable perspective for proving modern results (see, e.g., \cite{Wei21, BS22, BFHTW23}).  Moreover, divided differences are the tool that is used to implement Grothendieck polynomials for permutations computationally in Macaulay2 \cite{AGH+25}.

 Let \[\mathbb Z[\mathbf x,\mathbf y]=\mathbb Z[x_1,\ldots, x_n, y_1,y_2,\ldots, y_n].\] There is an action of $S_n$ on $\mathbb Z[\mathbf x,\mathbf y]$ which permutes the indices of the variables $x_i$ in the usual way and fixes each $y_i$.  Precisely, given $f\in\mathbb Z[\mathbf x,\mathbf y]$ and $w\in S_n$ we define 
 \[w\cdot f=f(x_{w(1)},\ldots, x_{w(n)},y_1,\ldots,y_n).\]  

 The \newword{divided difference operator} $\delta_i$ acts on $\mathbb Z[\mathbf x,\mathbf y]$ as follows: Given $f\in \mathbb Z[\mathbf x,\mathbf y]$, 
 \[\delta_i(f)=\frac{f-s_i\cdot f}{x_i-x_{i+1}}.\] Note that $\delta_i(f) \in \mathbb Z[\mathbf x,\mathbf y]$.  We also define the \newword{K-theoretic divided difference operator} $\pi_i$ by $\pi_i(f)=\delta_i(1-x_{i+1}f)$.
 
 Divided difference operators satisfy the same braid and commutation relations as do the simple reflections in the symmetric group, i.e.,
 \begin{itemize}
     \item $\delta_i\delta_j=\delta_j\delta_i$ and $\pi_i\pi_j=\pi_j\pi_i$ if $|i-j|>1$ 
     \item $\delta_i\delta_{i+1}\delta_i=\delta_{i+1}\delta_i\delta_{i+1}$ and $\pi_i\pi_{i+1}\pi_i=\pi_{i+1}\pi_i\pi_{i+1}$.
 \end{itemize}
 Additionally, $\delta_i^2=0$ and $\pi_i^2=\pi_i$.

 Write $w_0=n \, n-1 \, \ldots \, 1 \in S_n$.  Recall that, for $w \in S_n$, if $w>ws_i$ then $\pi_i(w) = ws_i$.  We record the usual characterizations of Schubert and Grothendieck polynomials in terms of divided difference operators.

 \begin{theorem}
 \label{theorem:grothendieckandschubertrecurrence}
 
 \underline{}
 
     \begin{enumerate}
         \item $\mathfrak G_{w_0}(\mathbf x,\mathbf y)=\prod_{\substack{i, j \geq 1 \\ i+j\leq n}} x_i+y_j-x_iy_j$ and $\mathfrak S_{w_0}(\mathbf x,\mathbf y)=\prod_{\substack{i, j \geq 1 \\ i+j\leq n}}x_i-y_j$. 
         \item If $w\in S_n$  then \[\mathfrak G_{\pi_i(w)}(\mathbf x,\mathbf y)=\pi_i(\mathfrak G_w(\mathbf x,\mathbf y))\] and  \[\mathfrak G_{\pi_i(w)}(\mathbf x)=\pi_i(\mathfrak G_w(\mathbf x)).\]
         \item If $w\in S_n$ so that $w>ws_i$, then   \[\mathfrak S_{\pi_i(w)}(\mathbf x,\mathbf y)=\delta_i(\mathfrak S_{w}(\mathbf x,\mathbf y))\]
         and \[\mathfrak S_{\pi_i(w)}(\mathbf x)=\delta_i(\mathfrak S_{w}(\mathbf x)).\]
         \item If $w\in S_n$ so that $w<ws_i$, then 
          \[\delta_i(\mathfrak S_{w}(\mathbf x,\mathbf y))=0=\delta_i(\mathfrak S_{w}(\mathbf x)).\]
     \end{enumerate}
     
 \end{theorem}

The statements of \cref{theorem:grothendieckandschubertrecurrence} are often taken as a definition.  That they are equivalent to the definitions we give in terms of $K$-polynomials and multidegrees follows from \cite[Theorem A]{KM05}.  See also \cite{Buc02,FR03}.
 
The primary goal of the remainder of this section is to show that ASM Grothendieck polynomials satisfy divided difference recurrences that are compatible with weak order on ASMs.  Along the way, we will contrast the behavior of ASM Grothendieck polynomials with that of ASM Schubert polynomials.

\begin{lemma}
\label{lem:asmpermjoingoesup}
    Let $A\in \asm(n)$ and assume that $A\not \in S_n$.  Write $\perm(A)=\{w_1,\ldots,w_k\}$.  Given $U\subseteq [k]$ with $U\neq \emptyset$,
    \[A<\bigvee_{i\in U}w_i.\]  Equivalently, $X_A\supsetneq \bigcap_{i\in U} X_{w_i}$, or, algebraically, $I_A \subsetneq \sum_{i \in U} I_{w_i}$.
\end{lemma}
\begin{proof}
   This is immediate from the definition of $\perm(A)$. 
\end{proof}

We now will prove that ASM Grothendieck polynomials satisfy a divided difference recurrence governed by weak order.

\begin{theorem}
\label{prop-DDO}
 Let $A\in \asm(n)$.  Then  $\pi_i(\mathfrak G_A(\mathbf x,\mathbf y))=\mathfrak G_{\pi_i(A)}(\mathbf x,\mathbf y)$, and $\pi_i(\mathfrak G_A(\mathbf x))=\mathfrak G_{\pi_i(A)}(\mathbf x)$.
\end{theorem}

\begin{proof}
If $A\in S_n$, the statement is true by \cref{theorem:grothendieckandschubertrecurrence}.  In particular the claim holds when $A=w_0$.  So fix $A\in \asm(n)-S_n$, and assume that $\pi_i(\mathfrak G_{A'}(\mathbf x,\mathbf y))=\mathfrak G_{\pi_i(A')}(\mathbf x,\mathbf y)$ for all $A'>A$.

Write $\perm(A)=\{w_1,\ldots,w_k\}$, and recall that $I_A=\bigcap_{j\in[k]}I_{w_j}$.  Given $\emptyset \neq U\subseteq [k]$, define $A_U=\bigvee_{j\in U} w_j$.     Thus by \cref{cor:kandtwistk},
    \begin{equation}
    \label{eq:inclusionexclusion}
        \mathfrak G_A(\mathbf x,\mathbf y)=\sum_{\emptyset \neq U\subseteq [k]} (-1)^{|U|-1}\mathfrak G_{A_U}(\mathbf x,\mathbf y).
    \end{equation}
    We apply $\pi_i$ to both sides of \cref{eq:inclusionexclusion} and obtain
    \begin{equation*}
        \pi_i(\mathfrak G_A(\mathbf x,\mathbf y))=\sum_{\emptyset \neq U\subseteq [k]} (-1)^{|U|-1}\pi_i(\mathfrak G_{A_U}(\mathbf x,\mathbf y)).
        \end{equation*}
        If $U\subseteq [k]$ with $U\neq \emptyset$, we have $A_U>A$ by \cref{lem:asmpermjoingoesup}.  Applying the inductive hypothesis, we have $\pi_i(\mathfrak G_{A_U}(\mathbf x,\mathbf y))=\mathfrak G_{\pi_i(A_U)}(\mathbf x,\mathbf y)$ for each such $U$.  Therefore,
        \begin{equation}\label{eq:pi_i(G)-expanded}
        \pi_i(\mathfrak G_A(\mathbf x,\mathbf y))=\sum_{\emptyset \neq U\subseteq [k]} (-1)^{|U|-1}\mathfrak G_{\pi_i(A_U)}(\mathbf x,\mathbf y).
        \end{equation}
Because $A_U=\bigvee_{j\in U} w_j$, we have $I_{A_U}=\sum_{j\in U} I_{w_j}$ by \cref{prop:idealsum}.  Then by \cref{prop:joindescentpart3}, $I_{\pi_i(A_U)}=\sum_{j\in U} I_{\pi_i(w_j)}$. By \cref{prop:intersectoperator}, $I_{\pi_i(A)} = \bigcap_{j \in [k]} I_{\pi_i(w_j)}$.  Hence, by \cref{cor:kandtwistk}, \begin{equation}\label{eq:G_pi_i-expanded}
     \mathfrak G_{\pi_i(A)}(\mathbf x,\mathbf y)=\sum_{\emptyset \neq U\subseteq [k]} (-1)^{|U|-1}\mathfrak G_{\pi_i(A_U)}(\mathbf x,\mathbf y).
\end{equation}
The result follows by combining \cref{eq:G_pi_i-expanded} and \cref{eq:pi_i(G)-expanded}.

The statement for single ASM Grothendieck polynomials follows by replacing the double ASM Grothendieck polynomials with the single versions throughout the proof or by specializing the double ASM Grothendieck polynomials appropriately.
\end{proof}

Having proved \cref{prop-DDO}, we are now prepared to prove that distinct ASMs have distinct ASM Grothendieck polynomials.

\begin{lemma}
\label{lemma:kpolynomialequal}
    Let $A,B\in \asm(n)$ so that $A\leq B$.  If $\mathfrak G_A(\mathbf x,\mathbf y)=\mathfrak G_B(\mathbf x,\mathbf y)$ then $A=B$.  Likewise, if $\mathfrak G_A(\mathbf x)=\mathfrak G_B(\mathbf x)$ then $A=B$.
\end{lemma}
\begin{proof}
     If $\mathfrak G_A(\mathbf x,\mathbf y)=\mathfrak G_B(\mathbf x,\mathbf y)$, then $\mathfrak G_A(\mathbf x)=\mathfrak G_B(\mathbf x)$.  If $\mathfrak G_A(\mathbf x)=\mathfrak G_B(\mathbf x)$, then we have $\hilb(S/I_A;\mathbf x)=\hilb(S/I_B;\mathbf x)$.  Since $A\leq B$ implies $I_A\subseteq I_B$, the equality of Hilbert series implies we have equality of ideals as well.  Hence $A=B$.
\end{proof}

\begin{proposition}
\label{prop:equal_grothendiecks_means_equal_ASMs}
    Let $A,B\in \asm(n)$.  The following are equivalent
    \begin{enumerate}
        \item $A=B$.
        \item $\mathfrak G_A(\mathbf x,\mathbf y)=\mathfrak G_B(\mathbf x,\mathbf y)$.
        \item $\mathfrak G_A(\mathbf x)=\mathfrak G_B(\mathbf x)$.
    \end{enumerate} 
\end{proposition}

\begin{proof} The implications (1) implies (2) implies (3) are trivial.

To see (3) implies (1), fix $A$ and $B$ satisfying $\mathfrak G_A(\mathbf x)=\mathfrak G_B(\mathbf x)$.  Suppose for contradiction that $A \neq B$, and assume that $A$ is minimal in strong order among counterexamples.  We will work by considering the set of descents of $A$, which either has at least two elements, exactly one element, or no elements.

Given $i\in [n-1]$, we have, by \cref{prop-DDO}, \[
\mathfrak G_{\pi_i(A)}(\mathbf x) =\pi_i( \mathfrak G_A(\mathbf x))  = \pi_i(\mathfrak G_B(\mathbf x)) =\mathfrak G_{\pi_i(B)}(\mathbf x).
\]

If $\pi_i(A)<A$, then by the minimality assumption on $A$, we must have $\pi_i(A)=\pi_i(B)$.  In particular, this implies $\rk_A(a,b)=\rk_B(a,b)$ for all $a,b\in[n]$ with $a\neq i$.  If $A$ has at least two distinct descents, then we have $\rk_A(a,b)=\rk_B(a,b)$ for all $a,b\in [n]$ which implies $A=B$.

Now suppose $A$ has exactly one descent, say in row $i$.  Then all of the essential cells of $A$ sit in row $i$, which implies that $A$ is the join of bigrassmannian permutations with descent $i$.  In particular, these are Grassmannian permutations (i.e., permutations with at most one descent).  It is a standard fact that the restriction of $S_n$ to the set of Grassmannian permutations whose unique descent is at $i$, together with the identity, forms a lattice under strong order (see, e.g., \cite[Section 2.4]{BB05}).  Thus, $A$ must actually be a Grassmannian permutation whose unique descent is $i$. 

We claim that $A \geq B$.  Because $\rk_A(a,b) = \rk_B(a,b)$ for all $a \neq i$, it suffices to show that $\rk_A(i,b) \leq \rk_B(i,b)$ for all $b \in [n]$.  Because $A \in S_n$, there exists a unique $j$ so that $A_{i,j} \neq 0$.  If $b<j$, then $\rk_A(i,b) = \rk_A(i-1,b) = \rk_B(i-1,b)$.  By \cref{lemma:cornerincrease}, $\rk_B(i,b) \geq \rk_B(i-1,b)  = \rk_A(i-1,b)$.  

Now consider $b \geq j$.  Because $A$ is a Grassmannian permutation with descent $i$, the unique nonzero entry of $A$ in row $i+1$ is to the left of column $j$.  Hence, $\rk_A(i,b) = \rk_A(i+1,b)-1 = \rk_B(i+1,b)-1$.  By \cref{lemma:cornerincrease}, $\rk_B(i,b) \geq \rk_B(i+1,b)-1 = \rk_A(i,b)$.  Thus, $A \geq B$, and so the result follows from \cref{lemma:kpolynomialequal} (or by the assumption that $A$ was chosen minimally).

Finally, if $A$ has no descents, then $A$ is the identity in $S_n$, and so $A \leq B$ and the result follows from \cref{lemma:kpolynomialequal}.
\end{proof}

\begin{remark}
    We note that the analog of \cref{prop:equal_grothendiecks_means_equal_ASMs} does not hold for twisted K-polynomials with respect to the standard $\mathbb N$-grading on $S=\field[z_{i,j}]$.  In particular, \[
    \mathfrak G_{w}(t,t,\ldots, t,0, 0, \ldots 0)=\mathfrak G_{w^{-1}}(t,t,\ldots, t,0, 0, \ldots 0)
    \] for all $w\in S_n$.
\end{remark}

Next we will show that a result analogous to \cref{prop-DDO} holds for ASM Schubert polynomials exactly when $\codim(X_{\pi_i(A)})<\codim(X_A)$.  The contrast between \cref{prop-DDO} and \cref{thm:codim-drops-iff-del-nontrivial-on-schub} gives an example of the K-theory behaving more predictably with respect to ASMs than cohomology, a consequence of the potential of ASM varieties to fail to be equidimensional.  We begin by recording a proposition.

\begin{proposition}
\label{prop:asmpolyschubertexpansion}
 Given $A\in \asm(n)$,   we have the following equalities of ASM Schubert polynomials: \[\mathfrak S_A(\mathbf x,\mathbf y)=\sum_{\substack{w\in \perm(A)\\\ell(w)=\codim(X_A)}}\mathfrak S_w(\mathbf x,\mathbf y)\] and
  \[\mathfrak S_A(\mathbf x)=\sum_{\substack{w\in \perm(A)\\\ell(w)=\codim(X_A)}}\mathfrak S_w(\mathbf x).\] 
\end{proposition}
\begin{proof}
    This is immediate from \cite[Theorem 1.7.1]{KM05} and \cite[Proposition 5.4]{Wei17}.
\end{proof}

\begin{theorem}\label{thm:codim-drops-iff-del-nontrivial-on-schub}
    Let $A\in \asm(n)$.  If $\codim(X_A)=\codim(X_{\pi_i(A)})$, then \[
    \delta_i(\mathfrak S_A(\mathbf x,\mathbf y))=0=\delta_i(\mathfrak S_A(\mathbf x)).
    \] Otherwise, $\delta_i(\mathfrak S_A(\mathbf x,\mathbf y))=\mathfrak S_{\pi_i(A)}(\mathbf x,\mathbf y)$, and $\delta_i(\mathfrak S_A(\mathbf x))=\mathfrak S_{\pi_i(A)}(\mathbf x)$.   
\end{theorem}
\begin{proof}
Using the expansion of $\mathfrak S_A(\mathbf x,\mathbf y)$ from \cref{prop:asmpolyschubertexpansion} and linearity of $\delta_i$, we have
\begin{align*}
\delta_i( \mathfrak S_A(\mathbf x,\mathbf y))&=\sum_{\substack{w\in \perm(A)\\\ell(w)=\codim(X_A)}}\delta_i(\mathfrak S_w(\mathbf x,\mathbf y)).\\
\end{align*}
If $\codim(X_A)=\codim(X_{\pi_i(A)})$, then, by \cref{cor:stepdowncodimbyatmost1}, all minimum length elements of $\perm(A)$, i.e., those satisfying $\ell(w) = \codim(X_A)$, have an ascent at $i$.  If $w$ has an ascent at $i$, then Part (4) of \cref{theorem:grothendieckandschubertrecurrence} says that $\delta_i(\mathfrak S_w(\mathbf x,\mathbf y))=0$.  Hence the sum is $0$, as desired.

If $\codim(X_A)\neq\codim(X_{\pi_i(A)})$, then, by Parts (3) and (4) of \cref{theorem:grothendieckandschubertrecurrence}, 

\[
\delta_i (\mathfrak S_A(\mathbf x,\mathbf y))=\sum_{\substack{w\in \perm(A)\\\ell(w)=\codim(X_A)}}\delta_i(\mathfrak S_w(\mathbf x,\mathbf y))=\sum_{\substack{w\in \perm(A)\\\ell(w)=\codim(X_A)\\ws_i<w}}\mathfrak S_{ws_i}(\mathbf x,\mathbf y).
\]
Then applying \cref{cor:perm-set-of-pi_i(A)-when-codim-drops} and \cref{prop:asmpolyschubertexpansion} yields
\[
\sum_{\substack{w\in \perm(A)\\\ell(w)=\codim(X_A)\\ws_i<w}}\mathfrak S_{ws_i}(\mathbf x,\mathbf y)
=\sum_{\substack{u\in \perm(\pi_i(A))\\\ell(u)=\codim(X_{\pi_i(A)})}}\mathfrak S_{u}(\mathbf x,\mathbf y)
=\mathfrak S_{\pi_i(A)}(\mathbf x,\mathbf y),
\]
as desired.

The statement for single ASM Schubert polynomials follows by replacing the double ASM Schubert polynomials with the single versions throughout the proof or by specializing the double ASM Schubert polynomials appropriately.
\end{proof}

\begin{example}\label{ex:schubert-forgets-small-component}
    Let $A = \begin{pmatrix}
        0 & 1 & 0 & 0\\
        0 & 0 & 1 & 0\\
        1 & -1 & 0 & 1\\
        0 & 1 & 0 & 0
    \end{pmatrix}$, in which case $X_A = X_{3412} \cup X_{2341}$.  Notice that $\codim(X_{3412}) = 4$ and $\codim(X_{2341}) = 3$.  Then $\pi_2(A) = \begin{pmatrix}
        0 & 1 & 0 & 0\\
        1 & -1 & 1 & 0\\
        0 & 0 & 0 & 1\\
        0 & 1 & 0 & 0
    \end{pmatrix}$ and $X_{\pi_2(A)} = X_{3142} \cup X_{2341}$ is equidimensional of codimension $3$.  We compute \begin{align*}
    \delta_2(\mathfrak S_A(\mathbf x,\mathbf y)) &=\delta_2(\mathfrak S_{2341}) = \delta_2((x_1-y_1)(x_2-y_1)(x_3-y_1)) \\
    &= \frac{(x_1-y_1)(x_2-y_1)(x_3-y_1)-(x_1-y_1)(x_3-y_1)(x_2-y_1)}{x_2-x_3} &= 0 
    \end{align*}
    while 
    \begin{align*}
   \mathfrak S_{\pi_2(A)}(\mathbf x,\mathbf y) &= \mathfrak S_{3142}+\mathfrak S_{2341}\\
   &= (x_1-y_1)(x_1-y_2)(x_3-y_2)+(x_1-y_1)(x_1-y_2)(x_2-y_1)\\
    &\hspace{5.39cm}+(x_1-y_1)(x_2-y_1)(x_3-y_1). \qedhere
    \end{align*} 
\end{example}

\begin{proposition}\label{prop:same-codim-iff-symmetric-Schubert}
    Let $A\in \asm(n)$.
    The following are equivalent:
    \begin{enumerate}
        \item $\codim (X_A)=\codim (X_{\pi_i(A)})$.
        \item The single ASM Schubert polynomial $\mathfrak S_A(\mathbf x)$ is symmetric in $x_i$ and $x_{i+1}$.
        \item The double ASM Schubert polynomial $\mathfrak S_A(\mathbf x,\mathbf y)$ is symmetric in $x_i$ 
        and $x_{i+1}$.
    \end{enumerate} 
\end{proposition}
\begin{proof}

Given $f\in \mathbb Z[\mathbf x,\mathbf y]$, we have that $f$ is symmetric in $x_i$ and $x_{i+1}$ if and only if $\delta_i(f)=0$. 
If $\codim (X_A)=\codim (X_{\pi_i(A)})$, then $\delta_i(\mathfrak S_A(\mathbf x))=\delta_i(\mathfrak S_A(\mathbf x,\mathbf y))=0$ by \cref{thm:codim-drops-iff-del-nontrivial-on-schub}.  Thus, both $\mathfrak S_A(\mathbf x)$ and $\mathfrak S_A(\mathbf x,\mathbf y)$ are symmetric in $x_i$ and $x_{i+1}$.

On the other hand, if $\codim (X_A) \neq \codim (X_{\pi_i(A)})$, then by \cref{thm:codim-drops-iff-del-nontrivial-on-schub}, we have $\delta_i(\mathfrak S_A(\mathbf x))=\mathfrak S_{\pi_i(A)}(\mathbf x)$.
By \cref{prop:asmpolyschubertexpansion}, $\mathfrak S_{\pi_i(A)}(\mathbf x)$ is a linear combination of Schubert polynomials.  Since $\perm(A)\neq \emptyset$, the sum is nonempty.  Because the set of Schubert polynomials indexed by permutations is linearly independent (see e.g., \cite[Corollary 6]{HPSW20}), $\mathfrak S_{\pi_i(A)}(\mathbf x)\neq 0$ and so $\mathfrak S_{A}(\mathbf x)$ is not symmetric in $x_i$ and $x_{i+1}$.  Since we obtain $\mathfrak S_{A}(\mathbf x)$ from $\mathfrak S_{A}(\mathbf x,\mathbf y)$ by setting each $y_j$ to $0$, $\mathfrak S_{A}(\mathbf x,\mathbf y)$ is also not symmetric in $x_i$ and $x_{i+1}$.
\end{proof}

\subsection{Derivatives} Define an operator $\nabla=\sum_{i=1}^n \frac{\partial}{\partial x_i}$.   Hamaker, Pechenik, Speyer, and Weigandt \cite{HPSW20} gave a formula for the application of $\nabla$ to single Schubert polynomials indexed by permutations, namely:
\[\nabla\left(\mathfrak S_w(\mathbf x)\right)=\sum_{s_iw<w}i\mathfrak S_{s_iw}(\mathbf x).\]
They used this formula to give a simple proof of Macdonald's reduced word formula for specializations of Schubert polynomials and also to prove a determinant conjecture of Stanley \cite{Stanley17}, which in turn implies that the weak order on the symmetric group has the strong Sperner property.  That the weak order has the strong Sperner property was first proved in \cite{GG20}.

In subsequent work, \cite{PSW24} generalized the derivative formula for Schubert polynomials to Grothendieck polynomials.  They used this formula to study the Castelnuovo--Mumford regularity of matrix Schubert varieties.  In this section, we state the analogous formula for ASM Grothendieck polynomials.

Write \[\des(A)=\{i \in [n-1] : A \text{ has a descent at } i\}.\] 
We define the \newword{major index} of $A\in \asm(n)$ to be $\maj(A)=\sum_{i\in \des(A)} i$.  We note that when $A\in S_n$, this definition agrees with the usual definition of major index.

Given $A\in \asm(n)$, we define $\pi_i^C(A)=\pi_i(A^T)^T$.  Equivalently, \begin{definition}
\label{def:pi_i^C}
    Given $A\in \asm(n)$ and $i\in[n-1]$, let 
    \[\pi_i^C(A)=\min\{B\in \asm(n):\rk_A(a,b)=\rk_B(a,b) \text{ for all } a,b\in [n] \text{ with } b\neq i\}.\]
\end{definition}

We compare to \cref{def:pi_i}. If one thinks of $\pi_i$ as operating on pairs of rows of an ASM, then $\pi_i^C$ is the analogous operator on pairs of columns.  In particular, for $w \in S_n$, $\pi_i^C(w) \in \{w, s_iw\}$.  
All lemmas and propositions proved for $\pi_i$ apply correspondingly to $\pi_i^C$.  In the proof of \cref{prop:derivatives}, below, we will once cite \cref{prop:intersectoperator} as justification for an equality concerning $\pi_i^C$.

Let  $E=\sum_{i=1}^n x_i\frac{\partial}{\partial x_i}$. The following proposition generalizes \cite[Theorem A.1]{PSW24}, which treats the case $A \in S_n$.

\begin{proposition}\label{prop:derivatives}
    Let $A\in \asm(n)$.  Then 
    \[\left(\maj(A^T)+\nabla-E\right)\mathfrak G_A(\mathbf x)=\sum_{\pi_i^C(A)<A}i\mathfrak G_{\pi_i^C(A)}(\mathbf x).\]
\end{proposition}
\begin{proof}

Note first that $i \in \des(A^T)$ if and only if $\pi_i^C(A)<A$, and so \begin{equation}\label{eq:maj(A^T)}
\maj(A^T) = \sum_{\pi_i^C(A) < A} i.
\end{equation}

We will proceed by induction on $\asm(n)$ under strong order. The longest permutation $w_0 \in S_n$ is the unique maximum element of this poset. The result holds for $w_0$ by \cite[Theorem A.1]{PSW24}.  Now take $A\neq w_0$, and assume the result holds for $B\in \asm(n)$ with $B>A$.

Let $\perm(A)=\{u_1,\ldots,u_k\}$.  By \cref{cor:kandtwistk} we have 
\[\mathfrak G_A(\mathbf x)=\sum_{\emptyset \neq U\subseteq [k]} (-1)^{|U|-1} \mathfrak G_{\vee \{u_j:j\in U\}}(\mathbf x).\]

For each $\emptyset \neq U \subseteq [k]$, we have $\vee \{u_j:j\in U\}>A$ by \cref{lem:asmpermjoingoesup}, and so by induction 
\begin{equation}
\label{eq:derivative-induction}
(\nabla-E)( \mathfrak G_{\vee \{u_j:j\in U\}}(\mathbf x)) 
= \sum_{i\in[n-1]} i \left(\mathfrak G_{\pi_i^C(\vee \{u_j:j\in U\})}(\mathbf x)-\mathfrak G_{\vee \{u_j:j\in U\}}(\mathbf x)\right).
\end{equation}

Applying $\nabla-E$ to both sides gives
\begin{align*}
(\nabla-E)(\mathfrak G_A(\mathbf x))&=(\nabla-E)\left(\sum_{\emptyset \neq U\subseteq [k]} (-1)^{|U|-1} \mathfrak G_{\vee \{u_j:j\in U\}}(\mathbf x) \right) \\
&=\sum_{\emptyset \neq U\subseteq [k]} (-1)^{|U|-1}(\nabla-E)( \mathfrak G_{\vee \{u_j:j\in U\}}(\mathbf x))  \\
&=\sum_{\emptyset \neq U\subseteq [k]} \left((-1)^{|U|-1}\sum_{i\in [n-1]} i (\mathfrak G_{\pi_i^C(\vee \{u_j:j\in U\})}(\mathbf x)-\mathfrak G_{\vee \{u_j:j\in U\}}(\mathbf x) )\right)  \\ &\hspace{10cm}\text{(\cref{eq:derivative-induction})}\\
&=\sum_{\emptyset \neq U\subseteq [k]} \left((-1)^{|U|-1}\sum_{i\in [n-1]} i (\mathfrak G_{\vee \{\pi_i^C(u_j):j\in U\}}(\mathbf x)-\mathfrak G_{\vee \{u_j:j\in U\}}(\mathbf x) )\right) \\
&\hspace{10cm} \text{(\cref{prop:intersectoperator})}\\
&=\sum_{i\in [n-1]}\left(i\sum_{\emptyset \neq U\subseteq [k]} (-1)^{|U|-1}  (\mathfrak G_{\vee \{\pi_i^C(u_j):j\in U\}}(\mathbf x)-\mathfrak G_{\vee \{u_j:j\in U\}}(\mathbf x) ) \right)\\
&=\sum_{i\in [n-1]}i  (\mathfrak G_{\pi_i^C(A)}(\mathbf x)-\mathfrak G_{A}(\mathbf x) ) \hspace{1cm} \text{(\cref{cor:kandtwistk} \mbox{ and }\cref{prop:intersectoperator})} \\
& =\left(\sum_{\pi_i^C(A)<A}i  (\mathfrak G_{\pi_i^C(A)}(\mathbf x)-\mathfrak G_{A}(\mathbf x) )\right)+\left(\sum_{\pi_i^C(A)=A}i  (\mathfrak G_{\pi_i^C(A)}(\mathbf x)-\mathfrak G_{A}(\mathbf x) )\right)\\
&=\left(\sum_{\pi_i^C(A)<A} i\mathfrak G_{\pi_i^C(A)}(\mathbf x)\right) - \sum_{\pi_i^C(A) < A} i \mathfrak G_A(\mathbf x)+0\\
&=\left(\sum_{\pi_i^C(A)<A} i\mathfrak G_{\pi_i^C(A)}(\mathbf x)\right) - \maj(A^T)\mathfrak G_A(\mathbf x) \hspace{1cm} \text{(\cref{eq:maj(A^T)}).} \qedhere
\end{align*}

\end{proof}

\section{Antichains of permutations}\label{s:antichains}

The purpose of this section is to generalize the results we have seen so far for ASM varieties to arbitrary unions of matrix Schubert varieties, which are indexed by antichains in $S_n$ under strong order.  The proofs of \cref{prop:intersectoperator2}, \cref{prop-DDO2}, and \cref{prop:derivatives2} rely on the results they generalize, each of which was proved using the rich combinatorial structure of $\asm(n)$.  In this way, this section demonstrates the value of studying $\asm(n)$, not merely as a method for gaining intuition about unions of matrix Schubert varieties, but as a valuable tool for proving these more general results.  

\subsection{Arbitrary unions of matrix Schubert varieties}

An \newword{antichain} in $S_n$ is a subset $\mathcal A\subseteq S_n$ so that for all $u,v\in \mathcal A$ with $u\neq v$, we have $u\not \leq v$ and $v \not \leq u$.  In other words, all pairs of distinct elements in $\mathcal A$ are incomparable under strong order.

 Let $\anti(n)$ be the set of nonempty antichains in $S_n$.  
Each $A\in \asm(n)$ can be uniquely identified by $\perm(A)\in \anti(n)$.  Thus, the map $A\mapsto \perm(A)$ defines an inclusion $\asm(n)\hookrightarrow \anti(n)$.  For this section, we will freely identify $A$ with the antichain $\perm(A)$.

If $\mathcal A\in \anti(n)$, we define $I_{\mathcal A}=\bigcap_{w \in \mathcal A} I_w$ and $X_{\mathcal A} = \mathbb{V}(I_\mathcal{A})$; hence, $X_{\mathcal A} = \bigcup_{w\in \mathcal A} X_w$, and $\codim(X_{\mathcal A})=\min\{\ell(w):w\in \mathcal A\}$.  In particular, the map $\mathcal A\mapsto X_{\mathcal A}$ defines a bijection between antichains of permutations and unions of matrix Schubert varieties. When $w\in S_n$, we have $\perm(w)=\{w\}$ and $X_{\{w\}}=X_w$ is a matrix Schubert variety.

Arbitrary unions of matrix Schubert varieties have been studied by Knutson \cite{Knu09}, by Bertiger \cite{Bertiger15}, 
and by this article's second and third authors \cite{KW23}.  

\subsection{Strong and weak order on antichains}

We define \newword{strong (Bruhat) order} on antichains by containment of varieties: $\mathcal A\leq \mathcal B$ if and only if $X_{\mathcal A}\supseteq X_{\mathcal B}$ if and only if $I_{\mathcal A}\subseteq I_{\mathcal B}$.  Restricted to ASMs (and to permutations), this is the usual strong order.  See \cref{fig:posetstrongantichain3} for the poset diagram of strong order on $\anti(3)$. 

For $i \in [n-1]$, we may also define operators $\pi_i$ which act on unions of matrix Schubert varieties componentwise, i.e.,
\[
\pi_i(X_{\mathcal A})=\bigcup_{w\in \mathcal A} X_{\pi_i(w)}.
\]  

Let $\mathcal{B}$ be the subset of $\{\pi_i(w): w \in \mathcal{A}\}$ consisting of its minimal elements.  Then $\mathcal{B} \in \anti(n)$, and $\bigcup_{w\in \mathcal A} X_{\pi_i(w)} = X_{\mathcal{B}}$.  In this way, the operation of $\pi_i$ on unions of matrix Schubert varieties in turn induces an operation on antichains: We let $\pi_i(\mathcal A)=\mathcal B$ if $\pi_i(X_{\mathcal A})=X_{\mathcal B}$.  If $\mathcal{A}=\perm(A)$ for some $A \in \asm(n)$, it follows from \cref{prop:intersectoperator} that $\pi_i(\mathcal{A})=\perm(\pi_i(A))$, i.e., the action of $\pi_i$ is the same whether we regard $\mathcal{A}$ as an ASM or as an antichain.  

We define the \newword{weak (Bruhat) order} on $\anti(n)$ to be the transitive closure of the covering relations $\mathcal A\prec \mathcal B$ if $\mathcal A\neq \mathcal B$ and $\mathcal A=\pi_i(\mathcal B)$ for some $i\in[n-1]$.  See \cref{fig:posetweakantichain3} for the poset diagram of weak order on $\anti(3)$.
  
\begin{figure}[!htb]
    \centering
    \begin{minipage}{.49\textwidth}
        \centering
\begin{tikzpicture}
\node (A) at (0,0) {$\{321\}$};
\node (B) at (-2,-1.2) {$\{231\}$};
\node (C) at (0,-2.4) {$\{231,312\}$};
\node (D) at (2,-1.2) {$\{312\}$};
\node (E) at (-2,-3.6) {$\{213\}$};
\node (F) at (0,-4.8) {$\{213,132\}$};
\node (G) at (2,-3.6) {$\{132\}$};
\node (H) at (0,-6) {$\{123\}$};

\draw[thick] (A) -- node[above]{} (B);
\draw[thick] (B) -- node[above]{} (C);
\draw[thick] (A) -- node[left]{} (D);
\draw[thick] (D) -- node[right]{}(C);
\draw[thick] (C) -- node[right]{}(E);
\draw[thick] (C) -- node[right]{}(G);
\draw[thick] (E) -- node[above]{} (F);
\draw[thick] (G) -- node[above]{} (F);
\draw[thick] (F) -- node[above]{} (H);
\end{tikzpicture}
\caption{The poset diagram of strong order on $\anti(3)$.}
\label{fig:posetstrongantichain3}
\end{minipage}
\begin{minipage}{.49\textwidth}
\centering
\begin{tikzpicture}
\node (A) at (0,0) {$\{321\}$};
\node (B) at (-2,-1.2) {$\{231\}$};
\node (C) at (0,-2.2) {$\{231,312\}$};
\node (D) at (2,-1.2) {$\{312\}$};
\node (E) at (-2,-3.6) {$\{213\}$};
\node (F) at (0,-4.6) {$\{213,132\}$};
\node (G) at (2,-3.6) {$\{132\}$};
\node (H) at (0,-6) {$\{123\}$};
\draw[->,thick] (A) -- node[above]{$\pi_1$} (B);
\draw[->,thick] (A) -- node[above]{$\pi_2$} (D);
\draw[->,thick] (B) -- node[left]{$\pi_2$} (E);
\draw[->,thick] (C) -- node[above]{$\pi_2$} (E);
\draw[->,thick] (C) -- node[above]{$\pi_1$} (G);
\draw[->,thick] (D) -- node[right]{$\pi_1$} (G);
\draw[->,thick] (E) to [bend right=25] node[left]{$\pi_1$} (H);
\draw[->,thick] (G) to [bend left=25] node[right]{$\pi_2$} (H);
\draw[->,thick] (F) to [bend right=25] node[left]{$\pi_1$} (H);
\draw[->,thick] (F) to  [bend left=25] node[right]{$\pi_2$} (H);
\end{tikzpicture}
\caption{The poset diagram of weak order on $\anti(3)$.}
\label{fig:posetweakantichain3}
\end{minipage}
\end{figure}

In $\anti(3)$, note that $\pi_1(\{213, 132\}) = \pi_2(\{213, 132\}) = \{123\}$. This example shows that it is no longer possible in this setting to uniquely identify weak order chains by the indices of the $\pi_i$'s.

\begin{remark}
If $\mathcal A=\perm(A)$ for some $A\in \asm(n)$, then the principle weak order ideal in $\anti(n)$ defined by $\mathcal A$ (i.e., $\{\mathcal B \in \anti(n): \mathcal B \preceq \mathcal A\}$) consists entirely of antichains that correspond to ASMs.  Likewise, given a singleton antichain (i.e., an antichain which corresponds to a permutation), the principal order ideal defined by this element consists entirely of permutations.  
\end{remark}

Our definition of weak order on $\anti(n)$ is motivated by the following extension of \cref{cor:pi_iorderpreserving} to arbitrary unions of matrix Schubert varieties.   
\begin{prop}\label{prop:intersectoperator2}
    Let $w_1, \ldots, w_r, u_1, \ldots, u_t \in S_n$ and $i \in [n-1]$.  If $\bigcap_{j=1}^r I_{w_j} \subseteq \bigcap_{k=1}^t I_{u_k}$, then $\bigcap_{j=1}^r I_{\pi_i(w_j)} \subseteq \bigcap_{k=1}^t I_{\pi_i(u_k)}$. 

    In the language of antichains, if $\mathcal{A}, \mathcal{B} \in \anti(n)$ with $\mathcal{A} \leq \mathcal{B}$ and $i \in [n-1]$, then $\pi_i(\mathcal{A}) \leq \pi_i(\mathcal{B})$.
\end{prop}
\begin{proof}
    We must show that $\bigcap_{j=1}^r I_{\pi_i(w_j)} \subseteq I_{\pi_i(u_k)}$ for all $k \in [t]$.  Fix $k \in [t]$.  Because $I_{u_k}$ is a prime ideal, $\bigcap_{j=1}^r I_{w_j} \subseteq I_{u_k}$ implies that there exists $j \in [r]$ so that $I_{w_j} \subseteq I_{u_k}$.  Hence, $I_{\pi_i(w_j)}\subseteq I_{\pi_i(u_k)}$ by \cref{cor:pi_iorderpreserving}, and so $\bigcap_{j=1}^r I_{\pi_i(w_j)} \subseteq I_{\pi_i(u_k)}$.  

    The final sentence is obtained by considering the case when the sets $\{w_j : j \in [r]\}$ and $\{u_k : k \in [t]\}$ each form an antichain, i.e., when the intersections are irredundant.
\end{proof}

Our next goal is to show that \cref{prop:intersectoperator2} extends to arbitrary intersections of unions of matrix Schubert varieties.  In order to do that, we will first give a lemma describing how to express these intersections as unions of matrix Schubert varieties.  To do so, we will work with their defining ideals.

\begin{lemma}\label{lem:rewrite-sum-of-antichain-ideals}
    Let $\mathcal{A}_1, \ldots, \mathcal{A}_r \in \anti(n)$.  Then  \begin{equation}\label{eq:which-Antichain-Ideal}
I_{\mathcal{A}_1}+\cdots+I_{\mathcal{A}_r} = \bigcap_{(w_{\alpha_1}, \ldots, w_{\alpha_r}) \in \mathcal{A}_1 \times \cdots \times \mathcal{A}_r} I_{w_{\alpha_1}}+\cdots+I_{w_{\alpha_r}}.
    \end{equation}
\end{lemma}
\begin{proof}
Set $I = I_{\mathcal{A}_1}+\cdots+I_{\mathcal{A}_r}$ and $J = \bigcap_{(w_{\alpha_1}, \ldots, w_{\alpha_r}) \in \mathcal{A}_1 \times \cdots \times \mathcal{A}_r} I_{w_{\alpha_1}}+\cdots+I_{w_{\alpha_r}}$.

The containment $I \subseteq J$ is a routine check from the definition of intersection.  

    It follows from the Frobenius splitting arguments of \cite[Section 1.2]{BK05} and \cite[Section 7.2]{Knu09} that $I$ is a radical ideal. Because each $I_{w_{\alpha_1}}+\cdots+I_{w_{\alpha_r}}$ is an ASM ideal, hence radical, and an intersection of radical ideals is radical, $J$ is also a radical ideal.  

     Hence, it suffices to show that every prime ideal containing $I$ also contains $J$.  Fix a prime ideal $P$ containing $I$.  Then $I_{\mathcal{A}_j} \subseteq I$ for each $j \in [r]$.  Because $P$ is prime and contains the intersection $I_{\mathcal{A}_j} = \bigcap_{w \in {\mathcal{A}_j}} I_w$, there must be some $w_{\alpha_j} \in {\mathcal{A}_j}$ so that $I_{w_{\alpha_j}} \subseteq P$.  By fixing such an $w_{\alpha_j}$ for each $j \in [r]$, we find an ideal $I_{w_{\alpha_1}}+\cdots+I_{w_{\alpha_r}}$ contained in $P$ and containing $J$. Hence $P$ contains $J$.
\end{proof}

The following proposition generalizes \cref{prop:joindescentpart3}.  

\begin{prop}\label{prop:pi-commutes-with-intersection-antichain-varieties}
    Let $\mathcal{A}_1, \ldots, \mathcal{A}_r \in \anti(n)$ and $i \in [n-1]$.  There exists $\mathcal{B} \in \anti(n)$ such that $I_{\mathcal{A}_1}+\cdots+I_{\mathcal{A}_r} = I_\mathcal{B}$.  Moreover, $I_{\pi_i(\mathcal{A}_1)}+\cdots+I_{\pi_i(\mathcal{A}_r)} = I_{\pi_i(\mathcal{B})}$.  
\end{prop}
\begin{proof}
Interpreted geometrically, \cref{lem:rewrite-sum-of-antichain-ideals} tells us that the intersection $\bigcap_{i \in [r]} X_{\mathcal{A}_i}$ of the unions $\bigcup_{w \in \mathcal{A}_i} X_w$ of matrix Schubert varieties is again a union of matrix Schubert varieties, specifically a union of the ASM varieties determined by the $I_{w_{\alpha_1}}+\cdots +I_{w_{\alpha_r}}$ (\cref{prop:idealsum}), each of which is in turn a union of matrix Schubert varieties (\cref{prop:I_A-intersection-of-schubs-in-perm}).  Hence, $\mathcal{B} \in \anti(n)$ as in the statement of the proposition exists.

Because each $I_{w_{\alpha_1}}+\cdots+I_{w_{\alpha_r}}$ is an ASM ideal and hence defines a union of matrix Schubert varieties, the second line of the following string of equalities follows from the definition of $\pi_i$ on $\anti(n)$, together with \cref{prop:intersectoperator2}, which allows us to evaluate $\pi_i$ on components of a union, even if that union is written with redundancies.  We compute
\begin{align*}
X_{\pi_i(\mathcal{B})}  = \pi_i(X_{\mathcal{B}}) &= \pi_i\left(\bigcup_{(w_{\alpha_1}, \ldots, w_{\alpha_r}) \in \mathcal{A}_1 \times \cdots \times \mathcal{A}_r} X_{w_{\alpha_1}}\cap \cdots \cap X_{w_{\alpha_r}}\right) & (\mbox{\cref{lem:rewrite-sum-of-antichain-ideals}})\\ 
&= \bigcup_{(w_{\alpha_1}, \ldots, w_{\alpha_r}) \in \mathcal{A}_1 \times \cdots \times \mathcal{A}_r} \pi_i(X_{w_{\alpha_1}} \cap\cdots \cap X_{w_{\alpha_r}}) & (\mbox{Definition of $\pi_i$})\\
&= \bigcup_{(w_{\alpha_1}, \ldots, w_{\alpha_r}) \in \mathcal{A}_1 \times \cdots \times \mathcal{A}_r} X_{\pi_i(w_{\alpha_1})} \cap \cdots \cap X_{\pi_i(w_{\alpha_r})} & (\mbox{\cref{prop:joindescentpart3}})\\
&= \bigcup_{(u_{\alpha_1}, \ldots, u_{\alpha_r}) \in \pi_i(\mathcal{A}_1) \times \cdots \times \pi_i(\mathcal{A}_r)} X_{u_{\alpha_1}}\cap\cdots \cap X_{u_{\alpha_r}} \\
& = X_{\pi_i(\mathcal{A}_1)}\cap \cdots \cap X_{\pi_i(\mathcal{A}_r)} & (\mbox{\cref{lem:rewrite-sum-of-antichain-ideals}}).
\end{align*}
Hence $I_{\pi_i(\mathcal{A}_1)}+\cdots+I_{\pi_i(\mathcal{A}_r)} = I_{\pi_i(\mathcal{B})}$, as desired.
\end{proof}

\begin{remark}
    One could know of the existence of $\mathcal{B} \in \anti(n)$ such that $I_{\mathcal{A}_1}+\cdots+I_{\mathcal{A}_r} = I_{\mathcal{B}}$ even without the explicit characterization of \cref{lem:rewrite-sum-of-antichain-ideals}.  Recalling that $I_{\mathcal{A}_1}+\cdots+I_{\mathcal{A}_r}$ is radical (\cite{BK05, Knu09}), one may then note that the variety it defines is invariant under the left-right action of the lower and upper Borel subgroups and then use \cite[Section 3]{Ful92} to infer that the variety must be a union of matrix Schubert varieties. 
\end{remark}

\subsection{Twisted K-polynomials of antichains}

Given $\mathcal A\in \anti(n)$, we write $\mathfrak G_{\mathcal A}(\mathbf x,\mathbf y)=\tilde{\mathcal K}(S/I_{\mathcal A};\mathbf x,\mathbf y)$.
Also, define \[\mathfrak G_{\mathcal A}(\mathbf x)=\mathfrak G_{\mathcal A}(x_1,\ldots, x_n,0,\ldots,0).\]

\begin{proposition}
\label{prop-DDO2}
    Let $\mathcal A\in\anti(n)$.  Then $\pi_i(\mathfrak G_{\mathcal A}(\mathbf x,\mathbf y))=\mathfrak G_{\pi_i(\mathcal A)}(\mathbf x,\mathbf y)$ and $\pi_i(\mathfrak G_{\mathcal A}(\mathbf x))=\mathfrak G_{\pi_i(\mathcal A)}(\mathbf x).$
\end{proposition}
\begin{proof}
    One follows the outline of the proof of \cref{prop-DDO}, with two modifications: (1) The citations of \cref{cor:pi_iorderpreserving} are replaced by the definition of the action of $\pi_i$ on $\anti(n)$, and (2) the statement for joins of the indexing permutations of components of $X_{\mathcal{A}}$ that arises from the application of \cref{cor:kandtwistk} is treated by appeal to \cref{prop-DDO} itself rather than via induction.
\end{proof}

The analog of \cref{prop:derivatives} holds in this setting as well.
We define $\des(\mathcal A)=\{i:\pi_i(\mathcal A)\neq \mathcal A\}$ and $\maj(\mathcal A)=\sum_{i\in \des(\mathcal A)} i$. 
Also, let $\mathcal A^T=\{u^{-1}:u\in \mathcal A\}$ and define $\pi_i^C(\mathcal A)=\pi_i(\mathcal A^T)^T$.  

\begin{proposition}
\label{prop:derivatives2}
    Let $\mathcal A\in \anti(n)$.  Then
    \[(\maj(\mathcal A^T)+\nabla-E)\mathfrak G_{\mathcal A}(\mathbf x)=\sum_{\pi_i^C(\mathcal A)<\mathcal  A}i\mathfrak G_{\pi_i^C(\mathcal A)}(\mathbf x).\]
\end{proposition}
\begin{proof}
One follows the outline of the proof of \cref{prop:derivatives}, with two modifications: (1) The citations of \cref{cor:pi_iorderpreserving} are replaced by the definition of the action of $\pi_i$ on $\anti(n)$, and (2) the statement for joins of the indexing permutations of components of $X_{\mathcal{A}}$ that arises from the application of \cref{cor:kandtwistk} is treated by appeal to \cref{prop:derivatives} itself rather than via induction.
\end{proof}

\begin{example}
    Let $\mathcal A=\{213,132\}$ and $\mathcal B=\{231,312\}$.  Note that $I_{213}+I_{132} = I_{\mathcal B}$ and that $I_{231}+I_{312} = I_{321}$.  We have 
    \begin{align*}
        \mathfrak G_{\mathcal A}(\mathbf x,\mathbf y)&=\mathfrak G_{213}(\mathbf x,\mathbf y)+\mathfrak G_{132} (\mathbf x,\mathbf y)-\mathfrak G_{\mathcal B} (\mathbf x,\mathbf y)\\
        &=\mathfrak G_{213}(\mathbf x,\mathbf y)+\mathfrak G_{132}(\mathbf x,\mathbf y)-\mathfrak G_{231}(\mathbf x,\mathbf y)-\mathfrak G_{312}(\mathbf x,\mathbf y)+\mathfrak G_{321}(\mathbf x,\mathbf y).
    \end{align*}
    In particular,
    \[\mathfrak G_{\mathcal A}(\mathbf x)=2x_1+x_2-2x_1x_2-x_1^2+x_1^2x_2.\]
We have $\maj(\mathcal A^T)=3$.  When we apply $\maj(\mathcal A^T)+\nabla-E$ to $\mathfrak G_{\mathcal A}(\mathbf x)$, we obtain 
\begin{align*}
    (\maj(\mathcal A^T)+\nabla-E)(\mathfrak G_{\mathcal A}(\mathbf x))&=\mathfrak G_{\pi_1^C(\mathcal A)}(\mathbf x)+2\mathfrak G_{\pi_2^C(\mathcal A)}(\mathbf x)\\
    &=3\mathfrak G_{123}(\mathbf x). \qedhere
\end{align*}
\end{example}

\begin{remark}
    Analogs of \cref{lem:permpiAascent}, \cref{cor:stepdowncodimbyatmost1}, \cref{cor:descent-of-A-implies-descent-in-perm}, \cref{cor:some-pi-drops-codim}, \cref{prop:asmpolyschubertexpansion}, \cref{thm:codim-drops-iff-del-nontrivial-on-schub}, \cref{prop:same-codim-iff-symmetric-Schubert} hold in $\anti(n)$ by following the given proofs with the replacement of \cref{cor:pi_iorderpreserving} by the definition of $\pi_i$ on $\anti(n)$; none of these arguments require already knowing the result for $\asm(n)$.
\end{remark}

\section*{Acknowledgements}

This project started as part of the Virtual Workshop for Women in Commutative Algebra and Algebraic Geometry, sponsored by the Fields Institute and organized by Megumi Harada and Claudia Miller.  The authors would like to thank Zach Hamaker and Vic Reiner for helpful discussions on weak order on ASMs.  They also thank Oliver Pechenik for helpful conversations.

\bibliographystyle{amsalpha} 
\bibliography{asm.bib}
\end{document}